\documentclass{article}
\usepackage[utf8]{inputenc} 
\usepackage{amsthm}
\usepackage{amssymb}
\usepackage{color}
\usepackage{amsfonts}
\usepackage{amsmath}
\usepackage{graphicx}
\usepackage{framed}
\usepackage{bbm}
\usepackage{setspace}
\usepackage{dina42e}

\usepackage{fancyhdr}

\begin{document}

\numberwithin{equation}{section}

\newtheorem{theorem}{Theorem}[section]
\newtheorem{lemma}[theorem]{Lemma}
\newtheorem{definition}[theorem]{Definition}
\newtheorem{proposition}[theorem]{Proposition}

\newtheorem{assumption}[theorem]{Assumption}

{
\theoremstyle{definition}
\newtheorem{annahme}{Assumption}
}

\newtheorem{remark}[theorem]{Remark}

{
\theoremstyle{definition}
\newtheorem{bemerkung}{Remark}
}

\renewcommand{\labelenumi}{$\roman{enumi})$}

\newcommand{\di}{\text{div}}
\newcommand{\eps}{\varepsilon}
\newcommand{\ve}{\bold v}
\newcommand{\Ve}{\textbf{V}}
\newcommand{\ue}{\textbf{u}}
\newcommand{\we}{\textbf{w}}
\newcommand{\Je}{\bold{ \tilde  J}}
\newcommand{\weight}[1]{\langle #1\rangle}
\newcommand{\into}{\int \limits_{\Omega}}
\newcommand{\intpo}{\int \limits_{\partial \Omega}}
\newcommand{\intt}{\int \limits_0^T}
\newcommand{\dx}{\mathit{dx}}
\newcommand{\dt}{\mathit{dt}}
\newcommand{\norm}[1]{|| #1 ||}
\newcommand{\R}{\mathbb{R}}

\fancyhead[ER]{}
\fancyhead[EL]{\thepage}
\fancyhead[OR]{\thepage}
\fancyhead[OL]{}
\fancyfoot[C]{}
\renewcommand{\headrulewidth}{0pt}   

\begin{titlepage}
\title{Existence of Weak Solutions for a Diffuse Interface Model
for Two-Phase Flow with Surfactants}
\author{Helmut Abels\footnote{Fakult\"at f\"ur Mathematik,
Universit\"at Regensburg,
93040 Regensburg,
Germany, e-mail: {\sf helmut.abels@mathematik.uni-regensburg.de}}, Harald Garcke\footnote{Fakult\"at f\"ur Mathematik,
Universit\"at Regensburg,
93040 Regensburg,
Germany, e-mail: {\sf harald.garcke@mathematik.uni-regensburg.de}},\ and Josef Weber\footnote{Fakult\"at f\"ur Mathematik,
Universit\"at Regensburg,
93040 Regensburg,
Germany}}
\end{titlepage}

\maketitle
\abstract{
Two-phase flow of two Newtonian incompressible viscous fluids with a soluble surfactant and different densities of the fluids
can be modeled within the diffuse interface approach. 
We consider a Navier-Stokes/Cahn-Hilliard type system coupled to
non-linear diffusion equations that describe the diffusion of the surfactant in the bulk phases as well as along the diffuse interface. Moreover, the surfactant concentration influences the free energy and therefore the surface tension of the diffuse interface. For this system existence of weak solutions globally in time for general initial data is proved. To this end a two-step approximation is used that consists of a regularization of the time continuous system in the first and a time-discretization in the second step.}

{\small\noindent
}


\section{Introduction and Main Result}

 Two-phase flow of two macroscopically immiscible, viscous,
incompressible fluids with a soluble surfactant is very relevant
in many applications and can be modelled by diffuse interface models.
{\bf Surf}ace {\bf act}ive {\bf a}gen{\bf{ts}} (surfactants)
are chemical substances which lower the surface tension at fluid interfaces.
They play an important role as
detergents, emulsifiers, wetting agents, dispersants and
foaming agents and hence mathematical modeling, analysis and numerical
computations involving  surfactants  have many potential applications.
In order to describe surfactants one has to model the fluid  flow in the
bulk regions away from the interface which is typically done with the help of the
Navier-Stokes equations. In addition, diffusion and advection of the surfactants has to be taken into account both at the interface and in the bulk. At the interface transport phenomena and force balances also involving  surface tension effects have to be modeled. As the presence of surfactants leads to a non-constant surface tension also Marangoni forces at the interface play an important role.

Classically the effects of surfactants at fluid interfaces are modeled within the context of sharp interface models, i.e., the interface is modeled as a hypersurface. We refer to \cite{article:BothePruss10, incoll:BothePrussSimonett, MR2213404}
for analytical results and to \cite{AHKN, BGNunsoluble, BGNsoluble, GanesanT12, article:JamesLowengrub04, article:KhatriTornberg11, article:LaiTsengHuang08, article:MuradogluTryggvason08, article:XuLiLowengrubZhao06}
for numerical results on sharp interface approaches for surfactant flow.

However, in the context of sharp interface approaches a change of topology, i.e.,
for example a splitting of a droplet or a reconnection of droplets, is difficult to describe theoretically and typically happens in an ad-hoc way numerically.
Therefore, in the last ten years diffuse interface approaches
(also called phase field approaches) have been introduced in the context
of surfactants in fluid  flow. 
We refer to \cite{preprint:EngblomDoQuangAmbergTornberg, MR3210738, 
article:LiKim12,article:LiuZhang10, article:TeigenLiLowengrubWangVoigt09, article:TeigenSongLowengrubVoigt11,
article:vanderSmanvanderGraaf06} for such approaches.

 In this paper
we will rely on the work of Garcke, Lam and Stinner \cite{MR3210738}
where for the first time   diffuse interface models were introduced which allowed for a physically sound energy inequality.
This fact is crucial for the analysis in this paper as it allows to derive important a priori estimates. The models in \cite{MR3210738} generalize
the thermodynamically consistent diffuse interface model for two-phase flow
with different densities introduced in \cite{MR2890451} to  situations
with surfactants.  
Existence of solutions for the model in \cite{MR2890451} has been shown
in \cite{MR3084319,MR3132421}. However, the case with surfactants turns out
to be much more involved due to the complex, nonlinear coupling of the
Navier-Stokes/Cahn-Hilliard system to the equations describing the surfactants.
This paper gives a first existence result for this general situation.

In the following we consider a diffuse interface model for such a
two-phase flow, where a partial mixing of the macroscopically immiscible
fluids is
taken into account and an order parameter, here in form of the
difference of volume fractions $\varphi$, is introduced. The interface
is no longer treated explicitly in form of a lower dimensional surface.
It is replaced by an interfacial region, where the order parameter
$\varphi$ is not close to $+1$ or $-1$. In a
sufficiently smooth situation the interfacial region has a  thickness
proportional to  $\eps>0$, where $\eps>0$ is a parameter in the system.

The model that we will analyze is a variant of the ``model C'' in \cite{MR3210738}. In the latter contribution different models for the case of soluble and insoluble surfactants were derived. Here we consider a model, where the surfactant is soluble in both fluids and accumulates at the interface. The model leads to the system
\begin{align}
\partial_t (\rho \ve) + \di ( \ve \otimes ( \rho \ve + \Je )) + \nabla p &- \di  ( 2 \eta (\varphi) D\ve  ) =  -\di ( \varepsilon \nabla \varphi \otimes \nabla \varphi   ) &&  \text{ in } Q , \label{equation_model_0}  \\
\di \ \ve &= 0  && \text{ in } Q , \label{equation_model_nabla_v}  \\
\partial_t^\bullet \left ( \frac{1}{\varepsilon} f (q) W (\varphi) + g (q) \right ) &= \di \left ( m( \varphi , q)  \nabla q \right )    && \text{ in } Q ,  \label{equation_model_1} \\
\partial_t^\bullet \varphi & = \di ( \tilde m (\varphi) \nabla \mu ) && \text{ in } Q  , \label{equation_model_2} \\
 - \varepsilon \Delta \varphi + h(q)  \frac{1}{\varepsilon} W ' (\varphi)  &= \mu && \text{ in } Q  ,  \label{equation_model_3} 
\end{align}
where $Q=\Omega\times (0,\infty)$, $\partial_t^\bullet:=\partial_t +\ve \cdot \nabla$ is the material time derivative and 
\begin{align}\label{eq:Je}
\Je &= \frac{\partial \rho (\varphi)}{\partial \varphi} (-  \tilde m (\varphi) \nabla \mu ) .
\end{align}
Throughout the paper $\Omega\subseteq \R^d$, $d=2,3$, is a bounded domain with $C^2$-boundary.
We close the system by prescribing initial values
\begin{align}\label{equation_model_inital_value_v}
\ve_{|t=0} &= \ve_0,  \quad \varphi_{|t=0} = \varphi_0, && \quad \nonumber \\
  \left ( \frac{1}{\varepsilon} f(q) W(\varphi) +g(q) \right )_{|t=0} &= \frac{1}{\varepsilon}f(q_0) W(\varphi_0) + g(q_0)   && \text{ in } \Omega , 
\end{align}
and boundary conditions
\begin{align}\label{equation_model_boundary_condition}
\ve_{| \partial \Omega} =   \partial_n \varphi_{| \partial \Omega} =  \partial_n q _{| \partial \Omega} =  \partial_n \mu _{| \partial \Omega} = 0    && \text{ on } \partial \Omega \times (0, \infty).
\end{align}
In this model, $\ve$ is the mean velocity of both fluids and $D \ve :=
\frac{1}{2} ( \nabla \ve + \nabla \ve^T )$. Moreover, $\varphi$ denotes
the order parameter, which is taken as the volume fraction difference of
both fluids, $\varepsilon > 0$ is a constant parameter related to the
``thickness" of the interfacial region and $p$ is the pressure.  The
viscosity of the mixture is denoted by $\eta (\varphi) > 0$ and $W\colon
\R\to \R$ is a potential, which is part of the free energy of the fluid
mixture and of double well shape with two global minima at $\pm 1$. The
chemical potential for the fluids is denoted by $\mu$, while $q$ is the
chemical potential for the surfactants.   
 Furthermore, $\tilde m (\varphi ) >0$ and $m (\varphi , q) > 0$ are the mobilities for the phase field and the surfactant diffusion equations, respectively.
The density $\rho$ depends on $\varphi$ and is given by
\begin{align*}
\rho (\varphi) = \frac{\tilde \rho_1 + \tilde \rho_2}{2} + \frac{\tilde \rho_2 - \tilde \rho_1}{2} \varphi \qquad \text{ if } \varphi \in [-1,1] .
\end{align*} 
In (\ref{equation_model_0}), 
 $\Je$ is a flux of the fluid density relative to $\rho \ve$.  The term $h(q)$ is  the square of the surface tension.  Furthermore, thermodynamic relations 
\begin{align}\label{eq:Relations}
 d' (q)=f' (q) q  , \quad
d(q) = h(q) + f(q) q  ,\quad
G' (q) = g'(q) q  .
\end{align}
are required. The first two equations imply $h'=-f$ and hence $d=h-h'q$. This means that $d$ is the Legendre transform of $h$
and this fact is related to the classical Gibbs relation from
thermodynamics, see \cite{AGLW}. 
In \cite{AGLW} it is discussed how the diffuse interface model \eqref{equation_model_0}-\eqref{equation_model_3} can be related to classical descriptions,
see \cite{article:BothePruss10, incoll:BothePrussSimonett,
 MR2213404},
   involving a sharp
interface.

The equations \eqref{equation_model_0} and \eqref{equation_model_nabla_v} describe the conservation of mass and linear momentum of the fluid mixture, \eqref{equation_model_1} describes the diffusion of the surfactants in the bulk of both fluids as well as along the diffuse interface and \eqref{equation_model_2}-\eqref{equation_model_3} describes the dynamics of the diffuse interface, the order parameter $\varphi$, respectively. Moreover,
$\eps\frac{|\nabla \varphi|^2}2+\frac{1}{\varepsilon} h(q) W (\varphi) $ is the free energy density of the fluid mixture and $\frac{1}{\varepsilon} f(q)q W (\varphi)+ G(q)$ is the free energy density of the surfactants at the diffuse interface (related to $f(q)q$) and in the bulk (related to $G(q)$). Hence the total free energy density is given by
\begin{equation*}
  \eps\frac{|\nabla \varphi|^2}2+\frac{1}{\varepsilon} d(q) W (\varphi) +G(q).
\end{equation*}

The total energy  of the system is given by
\begin{alignat}{1}
E_{tot}( \ve, \varphi, q) \label{equation_definition_total_energy}
&:=  \into \left (  \frac{1}{2} \rho(\varphi) |\ve|^2 +  \frac{\varepsilon}{2} |\nabla \varphi|^2  + \frac{d(q)}{\varepsilon} W (\varphi) + G(q) \right ) \dx .
\end{alignat}
Sufficiently smooth solutions of \eqref{equation_model_0}-\eqref{equation_model_3}, satisfying $\varphi(x,t)\in [-1,1]$ for all $x\in\Omega$, $t\in (0,\infty)$, together with the boundary conditions \eqref{equation_model_boundary_condition} satisfy the energy identity
\begin{alignat}{1}
  &\frac{d}{dt} E_{tot} (\ve,\varphi,q)
  = -\int_{\Omega}2\nu(\varphi)|D\ve|^2\,dx -\int_{\Omega}2\tilde{m}(\varphi)|\nabla \mu|^2\,dx -\int_{\Omega}m(\varphi,q)|\nabla q|^2\,dx 
\label{enineq}
\end{alignat}
for all $t\in (0,T)$. This is proved by testing \eqref{equation_model_0}, \eqref{equation_model_1}-\eqref{equation_model_3} with $\ve$, $q$, $\mu$, and $\partial_t \varphi$, respectively, integration by parts and using the relations \eqref{eq:Relations}. We refer to \cite[Section~2.2]{Dissertation_Weber} for the details.

To this end it is essential that
 $\frac{\partial \rho (\varphi)}{\partial \varphi}$ for all $\varphi\in [-1,1]$ is constant and therefore the continuity equation
\begin{align*}
 \partial_t \rho  (\varphi)  + \ve \cdot \nabla \rho (\varphi) + \di\, \Je = 0
\end{align*}
holds.

If $h(q)\equiv 1$, then
\eqref{equation_model_0}-\eqref{equation_model_nabla_v},
\eqref{equation_model_2}-\eqref{equation_model_3} is the diffuse
interface model for a two-phase flow of incompressible, viscous fluids
with different densities (without surfactants) that was derived in
\cite{MR2890451}. Existence of weak solutions to the model was shown in \cite{MR3084319} for non-degenerate mobility $m(\varphi)$ and singular free energy density $W$ and for degenerate mobility in \cite{
MR3132421}. Existence and uniqueness of strong solutions locally in time was shown in \cite{Dissertation_Weber}. We refer to \cite{Gigahandbook} for an overview of the different  models for two-phase flows without surfactants and analytic results.  

In the following analysis we will not be able to assure $\varphi(x,t)\in [-1,1]$ for all $x\in\Omega$, $t\in (0,\infty)$ since $\varphi$ solves a fourth order parabolic equation and no comparison principle is available. Moreover, we are not able to work with singular free energies or degenerate mobility as e.g. in \cite{MR3084319,MR3132421}. Therefore we extend $\rho\colon [-1,1]\to(0,\infty)$ as above to a continuously differentiable, bounded function $\rho\colon \R\to (0,\infty) $ with bounded derivative satisfying $\inf_{s\in \R}\rho(s)>0$.

In  a situation where the order parameter obtains values outside
of the interval $ [-1,1]$,  we obtain a modified continuity equation given by 
\begin{align}\nonumber
 &\partial_t \rho  (\varphi)  + \ve \cdot \nabla \rho (\varphi) + \di \Je  \\
&= \frac{\partial \rho (\varphi)}{\partial \varphi} \partial_t^\bullet \varphi + \di \Je \nonumber \\
 &= - \nabla \frac{\partial \rho (\varphi)}{\partial \varphi} \cdot \left ( \tilde m (\varphi) \nabla \mu \right ) = \frac{\partial \rho (\varphi)}{\partial \varphi} \di ( \tilde m (\varphi) \nabla \mu )  - \di \left ( \frac{\partial \rho (\varphi)}{\partial \varphi} \tilde m (\varphi) \nabla \mu \right )=:R \label{identity_R_tilde_R}
\end{align}
where we used (\ref{equation_model_2}). 
In order to preserve the energy identity above we modify \eqref{equation_model_0} to
\begin{align}
\partial_t (\rho \ve) + \di ( \ve \otimes ( \rho \ve + \Je )) + \nabla p &- \di  ( 2 \eta (\varphi) D\ve  ) -  \frac{R\ve}{2} =  -\di ( \varepsilon \nabla \varphi \otimes \nabla \varphi   )  \text{ in } Q_T. \label{equation_model_0'} 
\end{align}
Here $\frac{1}{2} R \ve$ describes the change in the kinetic energy associated to the source term $R$ in the modified continuity equation.
We note that this modified continuity equation and the additional source term $R$ is also used in \cite{Preprint150905663v1}. We also remark that the additional source term is zero for physically meaningful values $\varphi\in [-1,1]$.

With (\ref{equation_model_3}), equation (\ref{equation_model_0'}) can equivalently be written as
\begin{align*}
\partial_t (\rho \ve) + & \di ( \ve \otimes ( \rho \ve + \Je )) + \nabla p - \di  ( 2 \eta (\varphi) D\ve  ) -  \frac{R\ve}{2} =   \left ( \mu - \frac{h(q)}{\varepsilon} W' (\varphi) \right ) \nabla \varphi    &&  \text{ in } Q_T ,
\end{align*}
where a suitable scalar function is added to $p$, see  also
\cite{MR3210738}.

Before we state our main result we  define a  weak solution for \eqref{equation_model_nabla_v}-\eqref{equation_model_boundary_condition}, \eqref{equation_model_0'}:
\begin{definition}\label{definition_without_delta_weak_solutions}
Let $\bold v_0 \in L^2_\sigma (\Omega) $, $\varphi_0 \in H^2_n (\Omega):=\{u\in H^2(\Omega):\partial_n u|_{\partial\Omega}=0\} $, $q_0 \in L^2 (\Omega) $ be given. We call $(\bold v, \varphi, \mu, q)$ with the properties
\begin{align*}
\bold v & \in  L^2 (0,\infty; H^1_{0} (\Omega)^d ) \cap L^\infty (0,\infty; L^2_\sigma (\Omega)), 
&\varphi & \in L^\infty (0,\infty; H^1 (\Omega)) \cap L^2 (0,\infty; H^2 (\Omega)) , \\
\mu & \in L^2_{loc} ([0,\infty); H^1 (\Omega)) , 
&q & \in L^2 (0,\infty; H^1 (\Omega)) \cap L^\infty (0,\infty; L^2 (\Omega)) 
\end{align*}
a weak solution of \eqref{equation_model_nabla_v}-\eqref{equation_model_boundary_condition}, \eqref{equation_model_0'} if the following equations are satisfied:
\begin{align}
& - \int \limits_0^\infty  \into \rho \bold v \cdot \partial_t  \boldsymbol \psi \dx \dt -\int_\Omega \rho(\varphi_0)\ve_0\cdot\boldsymbol \psi|_{t=0}\dx - \int \limits_0^\infty \into ( \rho \bold v \otimes \bold v ) : \nabla \boldsymbol \psi \dx \dt   - \left \langle \frac{ R \bold v}{2}, \boldsymbol \psi \right \rangle   \label{eq:WeakNSt}\\\nonumber
& +  \int \limits_0^\infty \into 2 \eta (\varphi) D \bold v : D  \boldsymbol \psi \dx \dt- \int \limits_0^\infty \into ( \Je \otimes \bold v ) : \nabla \boldsymbol \psi \dx \dt  = \int \limits_0^\infty \into  \left ( \mu - \frac{h(q)}{\varepsilon} W' (\varphi) \right ) \nabla  \varphi \cdot \boldsymbol \psi \dx \dt 
\end{align}
for all $ \boldsymbol \psi \in C^\infty_{(0)} ([0,\infty); C^\infty_{0, \sigma} (\Omega))$ and
\begin{align}
   \int \limits_0^\infty \into \left ( f(q) W(\varphi) + g(q) \right ) \partial_t \phi \dx \dt &+ \into \left ( f(q_0) W(\varphi_0) + g(q_0) \right ) \phi|_{t=0} \dx  \nonumber \\\label{eq:WeakQ}
 \ \ \  +  \int \limits_0^\infty \left ( \frac{1}{\varepsilon} f(q) W(\varphi) + g(q) \right ) \bold v \cdot \nabla \phi \dx \dt  &=\int \limits_0^\infty \into m (\varphi, q) \nabla q \cdot \nabla \phi \dx \dt , \\\nonumber
 \int \limits_0^\infty \into \tilde m (\varphi) \nabla \mu \cdot \nabla \phi \dx \dt &= \int \limits_0^\infty \into \varphi \partial_t \phi \dx \dt + \into \varphi_0 \phi|_{t=0} \dx -  \int \limits_0^\infty \into \nabla \varphi \cdot \bold v \ \phi \dx \dt , \\\nonumber
\int \limits_0^\infty \into \mu \phi \dx \dt &= \int \limits_0^\infty \into \varepsilon \nabla \varphi \cdot \nabla \phi \dx \dt + \int \limits_0^\infty \into \frac{1}{\varepsilon} h(q) W' (\varphi) \phi \dx \dt 
\end{align}
for all $\phi \in C^\infty_{(0)} ([0,\infty); C^1 (\overline \Omega))$, where $\Je$ is defined by \eqref{eq:Je} and
\begin{align}\label{definition_R_delta_to_0}
 \left \langle \frac{ R \bold v}{2} , \boldsymbol \psi \right \rangle  := \frac12 \int \limits_0^\infty \into  \nabla \frac{\partial \rho}{\partial \varphi}(\varphi) \cdot ( \tilde m (\varphi) \nabla \mu) \bold v \cdot \boldsymbol \psi \dx \dt  .
\end{align}
Moreover, the energy inequality
\begin{align}\nonumber
&E_{tot} (\bold v (t), \varphi (t) , q (t)) 
  + \int \limits_s^t \into \left ( m(\varphi, q) |\nabla q|^2 + \tilde m (\varphi) |\nabla \mu|^2 + 2 \eta (\varphi) |D \bold v|^2 \right ) \dx \mathit{d \tau} \nonumber \\\label{energy_inequality}
&\qquad \leq E_{tot} ( \bold v (s), \varphi (s), q(s)) .
\end{align}
 has to hold for all $ t \in [s, \infty)$ and almost all $s \in [0, \infty)$ including $s=0$, where $E_{tot}$ is defined as in \eqref{equation_definition_total_energy}.
\end{definition}
We refer to Section~\ref{sec:Prelim} below for the definition of the function spaces above.\\
{\bf Remark:} We note that \eqref{eq:WeakQ} contains a weak formulation for the initial value of $f(q)W(\varphi)+g(q)$. More precisely using \eqref{eq:WeakQ}
and Assumption \ref{assumption12}, one can show
\begin{equation*}
  f(q)W(\varphi)+g(q)\in W^1_{\frac{4}{3}} (0, T ; H^{-1} (\Omega))\hookrightarrow C([0,T];H^{-1}(\Omega))
\end{equation*}
and $f(q)W(\varphi)+g(q)|_{t=0}= f(q_0)W(\varphi_0)g(q_0)$ in $H^{-1}(\Omega)$.
By this the initial value of $q$ is prescribed implicitly since $q\mapsto f(q)W(\varphi)+g(q)$ is strictly monotone, cf.\ Assumption~\ref{assumption12} below, and the initial value for $\varphi$ is determined by the evolution equation for $\varphi$. Finally, we note that \eqref{eq:WeakNSt} prescribes the initial value of $\rho\ve$ as $\rho(\varphi_0)\ve_0$. Divergence free test functions are sufficient for this because of the same observations as in \cite[Section~5.2]{MR3084319}.

To obtain the existence of weak solutions in the sense of Definition \ref{definition_without_delta_weak_solutions}, we make the following assumptions.
\begin{assumption}\label{assumption12}
We assume that $f \in C^\infty (\mathbb R)$ is monotonically increasing and $G \in C^2 (\mathbb R)$ is strictly convex such that
\begin{align*}
G' (q) = \begin{cases} < c_0 q & \text{if } q < 0, \\ = 0 & \text{if } q = 0, \\ > c_0 q & \text{if } q > 0 \end{cases}
\end{align*}
and $G' (q) = g' (q) q$ for all $q \in \mathbb R$. Moreover, we assume that there exists a constant $C > 0$ such that
\begin{align*}
|G(q)|  \leq C   ( |q|^2 + 1) , \qquad
|G'(q)|  \leq C  ( |q| + 1) 
\end{align*}
for all $q \in \mathbb R$. 
The functions $d,f, h , W, \tilde{m} \colon \R\to \R $ are continuously
differentiable, $m\colon \R^2\to \R$, $\eta\colon \R\to \R$ are continuous, the identities  \eqref{eq:Relations} are satisfied, and   $\Omega \subseteq \mathbb R^d$, $d= 2,3$, is a bounded domain with $C^2$-boundary. Moreover, there exist some constants $0 < c_1 <  c_2 < \infty$ such that
\begin{align*} 
& d (q)   >  c_1 ,  &&W(\varphi)  \geq 0 ,
\\
& f' (q) q  = d' (q) ,   && c_1 \leq  m( \varphi , q), \tilde m (\varphi),  \eta (q) \leq c_2
\end{align*}
for all $q, \varphi \in \mathbb R$. 
Furthermore, $h$ is concave and there exist constants $q_{min}, q_{max} \in \mathbb R$ with $q_{min} < q_{max}$ such that $d(q) \equiv const.$ for all $q \notin [q_{min} , q_{max}]$.
\\
For the growth of $W$ and $W'$ we assume that there exist some constants $C_1, C_2 , C_3 > 0$ such that
\begin{align}\label{growth_condition_W}
|W (a) | \leq C_1 ( |a|^3 + 1) ,  \quad |W' (a)| & \leq C_1 ( |a|^2 + 1) ,  \quad
W(a)  \geq C_2 |a| - C_3
\end{align}
holds for all $a \in \mathbb R$. If it holds $\frac{\partial\rho (\varphi) }{\partial \varphi} \not \equiv const$, then there exists a constant $C > 0$ and $0 < s < 1$ such that
\begin{align*}
|W' (a)| \leq C (|a|^s + 1)\qquad \text{for all }a\in \mathbb R.
\end{align*}
\end{assumption}
From these assumptions it follows that $g$ is strongly monotone due to
\begin{align*}
(g(a) - g(b)) (a - b) & = \int \limits_b^a g' (x) \dx (a-b) \geq c_0 (a-b)^2.
\end{align*}
Moreover, these assumptions imply $f' (q) = 0$ for all $ q \notin [q_{min} , q_{max}]$ as it holds  $f'(q) q = d' (q)$. In particular, we observe that $f$ is a bounded function.
Due to $h(q) = d(q) - f(q)q$, we can deduce that there exists a constant $C > 0$ such that
\begin{align}\label{growth_condition_h}
 |h(q)| & \leq C (|q| + 1) .
\end{align}
In applications physically relevant values for $q$ lie in a bounded
interval and hence the assumptions for the values of functions
for $|q|$ large give no restrictions
in practice. This is due to the fact that we can modify the functions
for large $|q|$ as we like.

Our main result is:
\begin{theorem}[Existence of weak solutions]\label{main_theorem_existence_of_weak_solutions}
~
\\
Let $0 < T < \infty$ and $\bold v_0 \in L^2_\sigma (\Omega)$, $\varphi_0 \in H^2_n (\Omega)$ and $q_0 \in L^2 (\Omega)$ be given. Then there exists a weak solution $(\bold v, \varphi, \mu, q)$ in the sense of Definition \ref{definition_without_delta_weak_solutions}. 
\end{theorem}

\noindent
{\bf Remark:} Due to the term 
$\partial_t^\bullet \left ( \frac{1}{\varepsilon} f (q) W (\varphi) + g
(q) \right )$ in (\ref{equation_model_1}) one of the main difficulties
in proving Theorem 1.3 is to show compactness with respect to $q$
in suitable approximating problems. The structural assumptions on
$d,f,g,h, G$ will first allow us to show an energy inequality 
related to (\ref{enineq}) and then enables us to show compactness of
$q$. Moreover, the treatment of the term $R$ is subtle. To the end we approximate the system in two steps. First we regularize the system in order to obtain $\partial_t\varphi \in L^2((0,\infty)\times\Omega))$. Then we solve the regularized system with the aid of an implicit time discretization.

\vskip 5mm
\noindent{\bf Acknowledgements:} The authors acknowledge support by the SPP 1506 ``Transport Processes
at Fluidic Interfaces'' of the German Science Foundation (DFG) through grant GA695/6-1 and GA695/6-2. The results are part of the third author's PhD-thesis~\cite{Dissertation_Weber}.

\section{Preliminaries} \label{sec:Prelim}

\textbf{Notation:} Let $X$ be a Banach space and let $X'$ be its dual
space. Then the duality product is given by  $\weight{x', x}_{X', X} = \weight{x', x} = x' (x)$, where $x' \in X'$ and $x \in X$.
The natural numbers without $0$ are denoted by $\mathbb N$ and we set $\mathbb N_0 := \mathbb N \cup \{0\}$. For $\bold a , \bold b \in \mathbb R^d$, we define $\bold a \otimes \bold b := (a_i b_j)_{i,j=1}^d$. If $\bold A, \bold B \in \mathbb R^{d \times d}$, then we set $A : B := \sum \limits_{i,j=1}^d A_{ij} B_{ij}$.

$\\$
\textbf{Function spaces:}
In the following, let $\Omega$ be a bounded domain with $C^2$-boundary. For $1 \leq p \leq \infty$ we denote by $L^p (\Omega)$ the usual Lebesgue-space equipped with its norm $\| \cdot \|_{L^p (\Omega)}$. The usual $L^p$-Sobolev space of order $k \in \mathbb N_0$ is denoted by $W^k_p (\Omega)$. For $p=2$ we denote $H^k(\Omega)=W^k_2(\Omega)$. If $\Omega=(0,T)$ the Banach-space valued variants are denoted by $L^p(0,T;X)$ and $W^k_p(0,T;X)$, respectively. Furthermore,
\begin{align*}
  L^p_{loc}([0,\infty);X)&= \{f\colon [0,\infty)\to X: f|_{(0,T)}\in L^p(0,T;X)\text{ for all }T>0 \},\\
 W^1_{p,loc}([0,\infty);X)&= \{f\in  L^p_{loc}([0,\infty);X): \tfrac{d}{dt} f\in L^p_{loc}([0,\infty;)X) \},
\end{align*}
and $C^\infty_{(0)}([0,\infty),X)$ denotes the space of smooth function $f\colon [0,\infty)\to X$ with a support that is compactly contained in $[0,\infty)$.

Moreover, we define $C^\infty_{0, \sigma} (\Omega) := \{ \bold u \in C^\infty_0 (\Omega)^d : \ \di\, \bold u  = 0\}$ and $L^2_\sigma (\Omega) := \overline{C^\infty_{0, \sigma} (\Omega)}^{\|\cdot\|_{L^2 (\Omega)}}$.
We set
\begin{align*}
V (\Omega) &:= H^2 (\Omega)^d \cap H^1_0 (\Omega)^d \cap L^2_\sigma (\Omega), \\
H^2_n (\Omega) &:= \left\{ f \in H^2 (\Omega) : \ \partial_n u_{| \partial \Omega } = 0 \text{ on } \Omega\right\}
\end{align*}
and for $m \in \mathbb R$ we define $$L^p_{(m)}  (\Omega) := \left\{ f \in L^p (\Omega) : \frac{1}{|\Omega|} \into f(x) \dx = m \right\}.$$ 
The operator $P_0$ is the orthogonal projection onto $L^2_{(0)} (\Omega)$ given by $P_0 f := f - \frac{1}{|\Omega|} \into f(x) \dx$ for all $f \in L^2 (\Omega)$.

For the proof of existence of weak solutions, we need the following compactness result.
\begin{theorem}\label{theorem_simon_5} 
Let $X \subseteq B \subseteq Y$ be Banach spaces with compact embedding
$X \hookrightarrow B$. Furthermore, let $1 \leq p \leq \infty$ and
assume that
\begin{enumerate}
\item $ F$ be bounded in $L^p (0, T; X)$, 
\item $\sup_{f\in F} \| \tau_h f - f \|_{L^p (0, T-h; Y)} \rightarrow 0$ as $ h \rightarrow 0$,
\end{enumerate}
where $(\tau_h f)(t) :=  f(t+h)$ for $h > 0$.
Then $F$ is relatively compact in $L^p (0,T;B)$ (and in $C([0,T];B)$ if $p = \infty$).
\end{theorem}
$\\$
The proof of this theorem can be found in \cite[Theorem 5]{MR916688}.


\section{Existence of Weak Solutions for the Surfactant Model}\label{section_existence_of_a_weak_solution}
We approximate the system \eqref{equation_model_1}-\eqref{equation_model_3}, \eqref{equation_model_0'} by the following equations:
\begin{align}
\partial_t (\rho \ve) + \di ( \ve \otimes ( \rho \ve + \Je )) + \nabla p &- \di  ( 2 \eta (\varphi) D\ve  )  \nonumber \\
 -  \frac{R\ve}{2} + \delta  \Delta^2 \ve&=  \left(\mu -\frac{h(q)}\eps W'(\varphi)\right)\nabla \varphi &&  \text{in } Q_T , \label{equation_delta_model_0}  \\
\di \ \ve &= 0  && \text{in } Q_T , \label{equation_delta_model_nabla_v}  \\
(\partial_t+\ve\cdot \nabla) \left ( \frac{1}{\varepsilon} f (q) W (\varphi) + g (q) \right ) &= \di \left ( m( \varphi , q)  \nabla q \right )    && \text{in } Q_T ,  \label{equation_delta_model_1} \\
\partial_t \varphi +\ve \cdot \nabla \varphi & = \di ( \tilde m (\varphi) \nabla \mu ) && \text{in } Q_T  , \label{equation_delta_model_2} \\
 - \varepsilon \Delta \varphi + h(q)  \frac{1}{\varepsilon} W ' (\varphi) + \delta \partial_t \varphi &= \mu && \text{in } Q_T  .  \label{equation_delta_model_3} 
\end{align} 
Due to the term $\delta \Delta^2 \ve$ in (\ref{equation_delta_model_0}), we need the additional boundary condition
\begin{align}\label{boundary_condition_laplace_v}
\Delta \ve_{|\partial \Omega} = 0 && \text{ on } \partial \Omega \times (0,T) .
\end{align}
We define a weak solution for the approximating system \eqref{equation_delta_model_0}-\eqref{equation_delta_model_3} together with initial and boundary conditions \eqref{equation_model_inital_value_v}-\eqref{equation_model_boundary_condition} and \eqref{boundary_condition_laplace_v} as follows.
\begin{definition}[Weak solution in the case $\delta > 0$]\label{definition_weak_solution}~\\
Let $\delta > 0$ and $\bold v_0 \in L^2_\sigma (\Omega) $, $\varphi_0 \in H^2_n (\Omega) $,  $q_0 \in L^2 (\Omega) $ be given. We call $(\bold v, \varphi, \mu, q)$ with the properties
\begin{align*}
\bold v & \in  L^2 (0,\infty; V (\Omega) ) \cap L^\infty (0,\infty; L^2_\sigma (\Omega)),  && \varphi  \in L^\infty (0,\infty; H^1 (\Omega)) \cap W^1_{2,loc} ([0,\infty); L^2 (\Omega)) , \\
\mu & \in L^2_{loc} ([0,\infty); H^1 (\Omega)) , && q  \in L^2_{loc} ([0,\infty); H^1 (\Omega)) \cap L^\infty (0,\infty; L^2 (\Omega))
\end{align*}
a weak solution of  \eqref{equation_delta_model_0}-\eqref{equation_delta_model_3} together with the initial and boundary conditions \eqref{equation_model_inital_value_v}-\eqref{equation_model_boundary_condition} and \eqref{boundary_condition_laplace_v} if $\varphi|_{t=0}=\varphi_0$ and the following equations are satisfied:
\begin{align}\label{definition_weak_solution_time_dependent_case_0}
&  - \int \limits_0^\infty  \into \rho \bold v \cdot \partial_t  \boldsymbol \psi \dx \dt - \into \rho(\varphi_0) \ve_0\cdot \boldsymbol \psi |_{t=0}\dx - \int \limits_0^\infty \into ( \rho \bold v \otimes \bold v ) : \nabla \boldsymbol \psi \dx \dt\nonumber \\& +  \int \limits_0^\infty \into 2 \eta (\varphi) D \bold v : D  \boldsymbol \psi \dx \dt   
 - \int \limits_0^\infty \into ( \Je \otimes \bold v) : \nabla \boldsymbol \psi \dx \dt  - \left \langle \frac{\tilde R \bold v}{2} , \boldsymbol \psi \right \rangle  \nonumber \\
&+ \delta \int \limits_0^\infty \into \Delta \bold v \cdot \Delta \boldsymbol \psi \dx \dt  = \int \limits_0^\infty \into  \left ( \mu - \frac{h(q)}{\varepsilon} W' (\varphi) \right ) \nabla  \varphi \cdot \boldsymbol \psi \dx \dt 
\end{align}
for all $ \boldsymbol \psi \in C^\infty_{(0)} ([0,\infty); V (\Omega) )$, where $\Je$ is defined as in \eqref{eq:Je} and
\begin{align}\label{definition_Rv_2_psi}
 \left \langle \frac{\tilde R \bold v}{2} , \boldsymbol \psi \right \rangle := & \frac12 \int \limits_0^\infty \into \partial_t \rho (\varphi) \bold v \cdot \boldsymbol \psi \dx \dt   -  \frac{1}{2}   \int \limits_0^\infty \into \left (  \rho (\varphi) \bold v  + \Je \right ) \cdot \nabla ( \bold v \cdot \boldsymbol \psi ) \dx \dt .
\end{align}
Moreover,
\begin{align}
   \int \limits_0^\infty \into \left (\frac1\eps f(q) W(\varphi) + g(q) \right ) \partial_t \phi \dx \dt&+ \into \left (\frac1\eps f(q_0) W(\varphi_0) + g(q_0) \right ) \phi|_{t=0} \dx \dt  \nonumber \\
  +  \int \limits_0^\infty \into \left ( \frac{1}{\varepsilon} f(q)
W(\varphi) + g(q) \right ) \bold v \cdot \nabla \phi \dx \dt
\label{definition_weak_solution_time_dependent_case_1}&=\int
\limits_0^\infty \into m (\varphi, q) \nabla q \cdot \nabla \phi \dx \dt
, \\
- \int \limits_0^\infty \into \tilde m (\varphi) \nabla \mu \cdot \nabla \phi \dx \dt &= \int \limits_0^\infty \into\partial_t \varphi \phi \dx \dt +  \int \limits_0^\infty \into( \bold v\cdot  \nabla \varphi  ) \phi \dx \dt ,  \label{definition_weak_solution_time_dependent_case_2} \\
\int \limits_0^\infty \into \mu \phi \dx \dt = \int \limits_0^\infty \into \varepsilon \nabla \varphi \cdot \nabla \phi \dx \dt &+ \int \limits_0^\infty \into (\frac{1}{\varepsilon} h(q) W' (\varphi)+\delta \partial_t\varphi) \phi \dx \dt 
   \label{definition_weak_solution_time_dependent_case_3}
\end{align}
for all $\phi \in C^\infty_{(0)} ([0,\infty); C^1 (\overline \Omega))$ and the energy inequality 
\begin{align}\label{delta_energy_estimate}
 \int \limits_s^t \into &  ( m(\varphi, q)  |\nabla q|^2 + \tilde m (\varphi) |\nabla \mu|^2 + 2 \eta (\varphi) |D\bold v|^2 + \delta |\Delta \bold v|^2  + \delta |\partial_t \varphi|^2  ) \dx \mathit{d \tau}  \nonumber \\
& + E_{tot} (\bold v(t) , \varphi (t), q(t))  \leq E_{tot} ( \bold v(s), \varphi (s), q(s))
\end{align}
has to hold for all $ t \in [s, \infty)$ and almost all $s \in [0, \infty)$ including $s=0$. 
\end{definition}
\begin{remark}
Note the difference between \eqref{definition_R_delta_to_0} and \eqref{definition_Rv_2_psi}. If all appearing terms are smooth enough, both definitions are equivalent due to the modified continuity equation \eqref{identity_R_tilde_R}. We use \eqref{definition_Rv_2_psi} because this representation allows us to derive an energy inequality. But as the term $\partial_t \rho (\varphi)$ appears in \eqref{definition_Rv_2_psi}, we also need to estimate this term. Therefore, we insert the additional term $\delta \partial_t \varphi$ in equation \eqref{definition_weak_solution_time_dependent_case_3} and approximate the initial system. Hence, we first need to solve the approximating system and then pass to the limit $\delta \rightarrow 0$. But due to the additional terms $\delta \partial_t \varphi$ and $\delta \Delta \ve$ we are then able to use \eqref{definition_R_delta_to_0} instead of   \eqref{definition_Rv_2_psi} in the proof of Theorem \ref{main_theorem_existence_of_weak_solutions}.
\end{remark}

In order to prove our main result we use:
\begin{theorem}[Existence of weak solutions for $\delta > 0$]\label{theorem_existence_of_weak_solutions}
~\\
Let $\bold v_0 \in L^2_\sigma (\Omega)$, $\varphi_0 \in H^2_n (\Omega)$ and $q_0 \in L^2 (\Omega)$ be given. Then there exists a weak solution $(\bold v, \varphi, \mu, q)$ of  \eqref{equation_delta_model_0}-\eqref{equation_delta_model_3} in the sense of Definition~\ref{definition_weak_solution}. Moreover, $\varphi \in L^2_{loc} ([0, \infty); H^2 (\Omega))$.
\end{theorem}

This theorem is proven in Section~\ref{Appendix} with the aid of an
approximation by a time-discrete problem for a time-step  size $h>0$ and
passing to the limit $h \rightarrow 0$. The arguments are a non-trivial
generalization to the case with surfactants of   the proof of the corresponding result in \cite{MR3084319}.

Theorem~\ref{theorem_existence_of_weak_solutions} yields the existence of weak solutions $(\ve^\delta, \varphi^\delta, \mu^\delta, q^\delta)$ of the approximating system \eqref{equation_delta_model_0}-\eqref{equation_delta_model_3} in the sense of Definition~\ref{definition_weak_solution}.
It remains to prove the existence of weak solutions for the system \eqref{equation_model_0}-\eqref{equation_model_3}.
To this end, we have to pass to the limit $\delta \rightarrow 0$.
We can assume w.l.o.g. $\into \varphi \dx = 0$.
By changing $\varphi$ by a constant and shifting $W$, we can always reduce to this case.
Moreover, we will assume $\eps=1$ for simplicity in the following. Furthermore, we split the equation
\begin{align}
\delta \partial_t \varphi^\delta - \Delta \varphi^\delta &= h(q^\delta) W' (\varphi^\delta) + \mu^\delta && \text{ in } (0,T) \times \Omega \label{equation_unsplitted_phi_delta_0}, \\
\partial_n \varphi^\delta_{|\partial \Omega} &= 0  && \text{ on } (0,T) \times \partial \Omega , \\
\varphi^\delta_{|t=0 } &= \varphi_0 && \text{ in } \Omega \label{equation_unsplitted_phi_delta_2} 
\end{align}
by considering the system
\begin{align}
\delta \partial_t \varphi^\delta_1 - \Delta \varphi^\delta_1 &= P_0 (h(q^\delta) W' (\varphi^\delta) + \Delta \varphi_0 ) && \text{ in } (0,T) \times \Omega, \label{equation_splitted_phi_phi_1}\\
\delta \partial_t \varphi^\delta_2 - \Delta \varphi^\delta_2 &= P_0 (\mu^\delta )  && \text{ in } (0,T) \times \Omega,  \\
 \partial_n \varphi^\delta_{1|\partial \Omega} =  \partial_n \varphi^\delta_{2|\partial \Omega} &= 0 && \text{ on } (0,T) \times \partial \Omega, \\
\varphi^\delta_{1|t=0} = \varphi^\delta_{2|t=0} &= 0 && \text{ in } \Omega, \\
\into \varphi^\delta_1 \dx = \into \varphi^\delta_2 \dx &= 0 .   \label{equation_splitted_phi_mean_value_phi_1_phi_2}
\end{align}
Note that, if $\varphi^\delta_1$ and $\varphi^\delta_2$ are solutions of \eqref{equation_splitted_phi_phi_1}-\eqref{equation_splitted_phi_mean_value_phi_1_phi_2}, then $\varphi^\delta = \varphi^\delta_1 + \varphi^\delta_2 + \varphi_0$ is a solution of \eqref{equation_unsplitted_phi_delta_0}-\eqref{equation_unsplitted_phi_delta_2} since $P_0(h(q^\delta) W' (\varphi^\delta) + \mu^\delta)=h(q^\delta) W' (\varphi^\delta) + \mu^\delta$ because of $\int_\Omega \partial_t \varphi dx=0$ due to \eqref{definition_weak_solution_time_dependent_case_2}.

In order to get uniform bounds for the solutions of \eqref{equation_splitted_phi_phi_1}-\eqref{equation_splitted_phi_mean_value_phi_1_phi_2} for $\varphi^\delta_1 $ and $\varphi^\delta_2$, we use the following lemma.
\begin{lemma}\label{lemma_phi_solves_equation_independent_of_delta}
Let $\delta > 0$ and $f \in L^p (0,T; L^q_{(0)} (\Omega))$ be given for $1 < p < \infty$,  $2 \leq q < \infty$ and $0 < T \leq \infty$. Then there exists a solution $\varphi \in L^p (0,T; W^2_q (\Omega))$ of
\begin{align*}
\delta \partial_t \varphi - \Delta \varphi &= f  && \text{ in } (0,T) \times \Omega , \\
\partial_n \varphi_{| \partial \Omega} &= 0 && \text{ on } (0,T) \times \partial \Omega, \\
\varphi_{|t=0} &= 0 && \text{ in } \Omega,
\end{align*}
which can be estimated by $\|\varphi\|_{L^p (0,T; W^2_q (\Omega))} \leq C \|f\|_{L^p (0,T; L^q (\Omega))}$ for a constant $C > 0$ independent of $\delta > 0$.
\end{lemma}
\begin{proof}
For the extension of $f$ on $(T, \infty)$ given by
\begin{align*}
\tilde f (t) := \begin{cases} f(t) & \text{ if } t \in (0,T) , \\ 0 & \text{ else}, \end{cases}
\end{align*}
 it holds $\tilde f \in L^p (0, \infty; L^q_{(0)} (\Omega))$. We consider the problem
\begin{align*}
\delta \partial_t \tilde \varphi - \Delta \tilde \varphi &= \tilde f && \text{ in } (0,\infty) \times \Omega , \\
\partial_n \tilde \varphi_{|\partial \Omega} &= 0 && \text{ on } (0,\infty) \times \partial \Omega ,  \\
\tilde \varphi_{|t=0} &= 0 && \text{ in } \Omega .
\end{align*}
Then we define $\psi_\delta (t) := \tilde \varphi (\delta t)$ and $\tilde f_\delta (t) := \tilde f (\delta t)$ and rewrite the system to
\begin{align*}
\partial_t \psi_\delta - \Delta  \psi_\delta &= \tilde f_\delta && \text{ in } (0,\infty) \times \Omega , \\
\partial_n \psi_{\delta|\partial \Omega} &= 0 && \text{ on } (0,\infty) \times \partial \Omega ,  \\
\psi_{\delta|t=0} &= 0 && \text{ in } \Omega .
\end{align*}
From \cite[Theorem 8.2]{MR2006641} it follows that for every $0 < T < \infty$ there exists a unique solution $\psi_\delta \in L^p (0,T; W^2_q (\Omega))$ and \cite[Theorem 2.4]{MR1225809} yields  $\psi_\delta \in L^p (0, \infty; W^2_q (\Omega))$ together with the estimate
\begin{align}\label{estimate_max_reg_help_problem_psi_delta_tilde_f_delta}
\|\psi_\delta\|_{L^p (0, \infty; W^2_q (\Omega))} \leq C \|\tilde f _\delta\|_{L^p (0, \infty; L^q (\Omega))} 
\end{align}
for a constant $C > 0$ independent of $\delta$. Here we used that $\sigma ( \Delta_{N} ) \subseteq (- \infty, 0)$  implies that $\Delta_N$ has negative exponential type, cf. \cite[Theorem 12.33]{MR2028503}, where 
\begin{align*}
\Delta_N : \mathcal D (\Delta_N) :=  W^2_{q,N} (\Omega) \cap L^q_{(0)} (\Omega)\subseteq L^q_{(0)} (\Omega) \rightarrow L^q _{(0)} (\Omega)
\end{align*}
is the Neumann-Laplace operator on $L^q_{(0)}(\Omega)$.
The result  (\ref{estimate_max_reg_help_problem_psi_delta_tilde_f_delta}) yields
\begin{align*}
\|\tilde \varphi\|_{L^p (0, \infty; W^2_q (\Omega))} \leq C \|\tilde f\|_{L^p (0, \infty; L^q (\Omega))}
\end{align*}
and therefore we obtain
\begin{align*}
\|\varphi\|_{L^p (0, T; W^2_q (\Omega))} \leq C \|f\|_{L^p (0, T; L^q (\Omega))}
\end{align*}
for a constant $C >0$ independent of $\delta$.
\end{proof}

\begin{proof}[Proof of Theorem \ref{main_theorem_existence_of_weak_solutions}]~ \\
Theorem \ref{theorem_existence_of_weak_solutions} yields the existence of weak solutions $(\ve^\delta, \varphi^\delta, \mu^\delta, q^\delta)$ in the sense of Definition \ref{definition_weak_solution} for every $\delta > 0$, where $\varphi^\delta \in L^2_{loc} ([0,\infty); H^2 (\Omega))$. From the energy inequality (\ref{delta_energy_estimate}) we get the following bounds:
\begin{enumerate} 
\item $ \qquad (\ve^\delta)_{\delta >0} \text{ is bounded in } L^\infty (0, \infty; L^2_\sigma (\Omega)) \cap  L^2 (0, \infty; H^1_0 (\Omega)^d) $,
\item $ \qquad (\nabla q^\delta)_{\delta >0} \text{ and }  (\nabla \mu^\delta)_{\delta >0} \text{ are bounded in } L^2 (0, \infty; L^2 (\Omega)^d)$,
\item $ \qquad (\nabla \varphi^\delta)_{\delta >0} \text{ is bounded in } L^{\infty} (0, \infty; L^2 (\Omega)^d) $,
\item $ \qquad (W(\varphi^\delta))_{\delta >0} \text{ and }(G(q^\delta))_{\delta >0} \text{ are bounded in } L^{\infty} (0,\infty ; L^1 (\Omega)) $. 
\end{enumerate}
Since $\tilde R$ and $R$ also depend on $\delta $, we write $\tilde
R^\delta $ and $R ^\delta$ instead. Choosing $\phi= \frac{\ve}2 \cdot
\boldsymbol \psi \frac{\partial\rho}{\partial \varphi}(\varphi)$ in \eqref{definition_weak_solution_time_dependent_case_2} one can show
\begin{align*}
 \left \langle \frac{R^\delta \ve^\delta}{2} , \boldsymbol \psi \right \rangle &=   \left \langle \frac{ \tilde R^\delta \ve^\delta}{2} ,  \boldsymbol \psi \right \rangle\qquad \text{ for all }\boldsymbol \psi \in C^\infty_{(0)} ([0,\infty); C^\infty_{0, \sigma} (\Omega))
\end{align*}
 in a straightforward manner.
Thus we use (\ref{definition_R_delta_to_0}) instead of (\ref{definition_Rv_2_psi}) from now on. 
From (\ref{definition_weak_solution_time_dependent_case_1}) it follows that the mean value of $\varphi^\delta (t) $ is constant and from (\ref{definition_weak_solution_time_dependent_case_2}) we obtain that the mean value of $\mu^\delta (t)$ is bounded in $L^\infty (0, \infty)$. Consequently, there exist subsequences such that
\begin{enumerate}
\item $ \qquad \ve^\delta \rightharpoonup \ve $ in $L^2 (0, \infty ; H^1_0 (\Omega)^d )$,
\item $ \qquad \ve^\delta \rightharpoonup^* \ve $ in $ L^\infty (0, \infty; L^2_\sigma (\Omega) ) \cong ( L^1 (0, \infty; L^2_\sigma (\Omega) ))'$,
\item $ \qquad q^\delta \rightharpoonup q$ in $L^2 (0, T; H^1 (\Omega))$ for all $0<T<\infty$, 
\item $ \qquad q^\delta \rightharpoonup^* q$ in $L^\infty (0, \infty; L^2 (\Omega)) \cong ( L^1 (0,\infty; L^2 (\Omega)))'$,
\item $ \qquad \varphi^\delta \rightharpoonup^* \varphi$ in $L^{\infty} (0, \infty; H^1 (\Omega)) \cong (L^1 (0,\infty ; H^1 (\Omega)))'$,
\item $ \qquad \mu^\delta \rightharpoonup \mu$ in $L^2 (0, T; H^1 (\Omega))$ for all $0<T<\infty$. 
\end{enumerate}
 Using the growth conditions \eqref{growth_condition_W} for $W'$ and (\ref{growth_condition_h}) for $h$, we get the boundedness of $h(q^\delta) \in L^2 (0,T; L^6 (\Omega))$ and $W' (\varphi^\delta ) \in L^\infty (0,T; L^3 (\Omega))$ for every $T\in (0,\infty)$.
 Applying elliptic regularity theory to
\begin{align*}
\varepsilon \Delta \varphi^\delta = \frac{1}{\varepsilon} h(q^\delta) W ' (\varphi^\delta) - \mu^\delta + \delta \partial_t \varphi^\delta =: f_1 ^\delta
\end{align*}
cf.\ (\ref{equation_delta_model_3}), we obtain the boundedness of $(\varphi^\delta)_{\delta > 0}$ in $L^2 (0,T; H^2 (\Omega))$ since the right-hand side $f_1^\delta $ is bounded in $L^2 (Q_T)$ for every $T\in(0,\infty)$. Here we used the energy inequality (\ref{delta_energy_estimate}) to estimate the boundedness of $(\sqrt \delta \partial_t \varphi^\delta)_{\delta > 0}$ in $L^2 (Q_T)$.

Due to (\ref{equation_delta_model_2}) we deduce that $(\partial_t \varphi^\delta)_{\delta > 0}$ is bounded in $L^2 (0, \infty; H^{-1} (\Omega))$. More precisely, we use that $\ve^\delta \cdot \nabla \varphi^\delta$ is bounded in $ L^2 (0,\infty; L^\frac{3}{2} (\Omega))$ because of the boundedness of $ \ve^\delta$ in $L^\infty (0,\infty; L^2 (\Omega)$ and $\nabla \varphi^\delta $ in $L^2 (0,T; L^6 (\Omega))$ for every $0<T<\infty$.
Hence the Lemma of Aubin-Lions 
yields
\begin{align*}
\varphi^\delta \rightarrow \varphi \qquad \text{ in } L^p (0,T; L^2 (\Omega))
\end{align*}
for every $1 \leq p < \infty$ and $0<T<\infty$. 

For the proof of precompactness of $(q^\delta)_{\delta > 0}$ in $L^2 (0,T; L^2 (\Omega))$ we show that $(q^\delta )_{ \delta > 0}$ fulfills the assumptions of Theorem \ref{theorem_simon_5} with $ X = H^1 (\Omega)$, $B = L^2 (\Omega)$ and $Y = L^2 (\Omega)$.
First, we note that $(q^\delta)_{\delta > 0}$ is bounded in $L^2 (0,T; H^1 (\Omega))$. 
To prove condition ii) we show the estimate
\begin{align}\label{delta_estimate_compactness_q_plus_lambda_help}
\left ( \int \limits_0^{T-\lambda} \|q^\delta (t+\lambda) - q^\delta (t)\|^2_{L^2 (\Omega)} \dt \right ) ^{\frac{1}{2}} \leq C \lambda^\frac{1}{4} 
\end{align}
for any $\lambda \in (0,T)$ and a constant $ C > 0$ independent of $\lambda$ and $\delta$.
We define
\begin{align*}
 F^\delta (\varphi^\delta, q^\delta) := \frac{1}{\varepsilon} f(q^\delta) W(\varphi^\delta) + g(q^\delta) , 
&&  f^\delta (t) :=  F^\delta (\varphi^\delta (t), q^\delta (t)) .
\end{align*}
Then $ F^\delta (\varphi^\delta, \cdot )$ is strongly monotone since $g$ is strongly monotone and $f$ is monotone. Thus there exists a constant $C > 0$ such that
\begin{align*}
\left (  F^\delta (\varphi^\delta (t) , q^\delta (t + \lambda ) ) -  F^\delta (\varphi^\delta (t) , q^\delta (t) ) \right ) ( q^\delta (t + \lambda) - q^\delta (t) ) \geq C \left |  q^\delta (t+ \lambda) - q^\delta (t) \right |^2
\end{align*}
for every $t \in (0,T)$.
Integrating this inequality over the domain $\Omega\times (0,T-\lambda)$ yields
\begin{align}\label{delta_inequality_qNplus_s_qN_estimate}
 &C   \int \limits_0^{T-\lambda} \into  |q^\delta (t+  \lambda) - q^\delta (t)|^2  \dx \dt  \leq  \int \limits_0^{T - \lambda} \into  \left (  F^\delta (\varphi^\delta   , q^\delta _{\lambda + } ) -  F^\delta (\varphi^\delta _{\lambda + } , q^\delta_{\lambda +}) \right ) (q^\delta _{\lambda + } - q^\delta  )  \dx \dt \nonumber \\
&\qquad\qquad  +  \int \limits_0^{T - \lambda} \into  \left (  F^\delta (\varphi^\delta _{\lambda + }  , q^\delta_{\lambda + } ) -  F^\delta (\varphi^\delta  , q^\delta ) \right ) (q^\delta _{\lambda + } - q^\delta  ) \dx \dt  ,
\end{align}
where we introduce the notation $f_{\lambda +} (t) := f(t + \lambda)$.
The first term in  (\ref{delta_inequality_qNplus_s_qN_estimate}) can be estimated by
\begin{align*}
 &  \int \limits_0^{T - \lambda} \into  \left (  F^\delta (\varphi^\delta   , q^\delta _{\lambda + } ) -  F^\delta (\varphi^\delta _{\lambda + } , q^\delta_{\lambda +}) \right ) (q^\delta _{\lambda + } - q^\delta  ) \dx \dt
\\ &\ \ \ \leq C \int \limits_0^{T- \lambda} \into \left | W (\varphi^\delta ) - W (\varphi^\delta _{\lambda +}) \right | \ \left | q^\delta _{\lambda +} - q^\delta \right | \dx \dt .\\
& \ \ \  \leq C \int \limits_0^{T- \lambda} \into   \left | \varphi^\delta _{\lambda +} - \varphi^\delta  \right |  \left ( |\varphi^\delta |^2 + |\varphi^\delta _{\lambda +}|^2 +1 \right )  \  \left | q^\delta _{\lambda +} - q^\delta \right | \dx \dt  \leq C (T) \lambda^{\frac{1}{4}} .
\end{align*}
Here we used the growth condition \eqref{growth_condition_W} for $W'$ and the boundedness of $f$. Moreover, we used the boundedness of $\varphi^\delta \in L^\infty (0, T; L^6 (\Omega))$, $q^\delta \in  L^2 (0,T; L^6 (\Omega))$ and 
\begin{align}\label{delta_inequality_varphi_N_difference_l2}
\underset{0 \leq t \leq T- \lambda }{\sup} \|\varphi^\delta (t + \lambda) - \varphi^\delta (t)\|_{L^2 (\Omega)} \leq C(T) \lambda^{\frac{1}{4}} \quad \text{for all }\lambda\in (0,T]
\end{align}
for a constant $C(T) > 0$ depending on $T>0$. The latter inequality is proved as follows: Since $(\varphi^\delta)_{\delta > 0} $ is bounded in $L^\infty (0 , T; H^1_{(0)} (\Omega)) \cap W^1_2 (0, T ; H^{-1 }_{(0)} (\Omega))$ and 
$$
W^1_2 (0,T; H^{-1}_{(0)} (\Omega)) \hookrightarrow C^\frac{1}{2} ([0,T]; H^{-1}_{(0)} (\Omega)),
$$ 
we have
\begin{align*}
\underset{t \in [0, T-  \lambda )}{\sup} \| \varphi^\delta ( t +  \lambda) -  \varphi^\delta ( t )\|_{H^{-1} (\Omega)} \leq C \lambda ^{\frac{1}{2}}\qquad \text{for all }\lambda\in (0,T].
\end{align*}
Now (\ref{delta_inequality_varphi_N_difference_l2}) follows from $\|f\|_{L^2 (\Omega)} \leq \|f\|_{H^1_{(0)} (\Omega)}^\frac{1}{2}  \|f\|_{H^{-1}_{(0)} (\Omega)}^\frac{1}{2}$ for all $f \in H^1_{(0)} (\Omega)$.
To estimate the second term in (\ref{delta_inequality_qNplus_s_qN_estimate}), we use equation (\ref{equation_delta_model_1}). Then we obtain
\begin{align*}
& \int \limits_0^{T - \lambda} \into  \left (  F^\delta (\varphi^\delta _{\lambda + }  , q^\delta_{\lambda + } ) -  F^\delta (\varphi^\delta  , q^\delta ) \right ) (q^\delta _{\lambda + } - q^\delta  ) \dx \dt = \int \limits_0^{T - \lambda} \into  \left (  f^\delta _{\lambda + } - f^\delta \right ) (q^\delta _{\lambda + } - q^\delta  ) \dx \dt  \\
& = \int \limits_0^{ T - \lambda} \into \int \limits_t^{t + \lambda} \partial_\tau f^\delta (\tau) \mathit{d \tau} (q^\delta_{\lambda +} - q^\delta ) \dx \dt .
\end{align*}
Using (\ref{equation_delta_model_1}) yields
\begin{align*}
&  \int \limits_0^{T - \lambda} \into  \left (  F^\delta (\varphi^\delta _{\lambda + }  , q^\delta_{\lambda + } ) -  F^\delta (\varphi^\delta  , q^\delta ) \right ) (q^\delta _{\lambda + } - q^\delta  ) \dx \dt  \\
& \leq  \int \limits_0^{T-\lambda} \into \int \limits_t^{t + \lambda}  \left  | m (\varphi^\delta (\tau), q^\delta (\tau) ) \nabla q^\delta (\tau )  \right | \mathit{d \tau} |\nabla q^\delta_{\lambda +} - \nabla q^\delta| \dx \dt \\
& \ \ \  \ +   \int \limits_0^{T-\lambda} \into   \int \limits_t^{t + \lambda}  \left | g (q^\delta (\tau ) \ve^\delta (\tau ) \right | \mathit{ d \tau}   |\nabla q^\delta_{\lambda +} - \nabla q^\delta| \dx \dt \\
& \ \ \ \ + \int \limits_0^{T- \lambda } \into  \int \limits_t^{t + \lambda}  \left | \frac{1}{\varepsilon} f (q^\delta (\tau)) W (\varphi^\delta (\tau )) \ve^\delta (\tau )\right | \mathit{d \tau}  |\nabla q^\delta_{\lambda +} - \nabla q^N| \dx \dt .
\end{align*}
The first term can be estimated by
\begin{align*}
& \int \limits_0^{T-\lambda} \into \int \limits_t^{t + \lambda}  \left  | m (\varphi^\delta (\tau), q^\delta (\tau) ) \nabla q^\delta (\tau )  \right | \mathit{d \tau} |\nabla q^\delta_{\lambda +} - \nabla q^\delta| \dx \dt \\
& \leq C \int \limits_0^{T- \lambda} \int \limits_t ^{t + \lambda}  \|\nabla q^\delta (\tau)\|_{L^2 (\Omega)}  \mathit{d \tau}  \|\nabla q^\delta_{\lambda +} - \nabla q^\delta \|_{L^2 (\Omega)} \dt \\
& \leq C \int \limits_0^{T- \lambda } \lambda^\frac{1}{2} \|q^\delta\|_{L^2 (0,T; H^1 (\Omega))}   \|\nabla q^\delta_{\lambda +} - \nabla q^\delta \|_{L^2 (\Omega)} \dt   \leq C (T) \lambda^\frac{1}{2} .
\end{align*}
For the third term we obtain
\begin{align*}
& \int \limits_0^{T-\lambda} \into  \int \limits_t^{t + \lambda}  \left | \frac{1}{\varepsilon} f (q^\delta (\tau )) W (\varphi^\delta (\tau )) \ve^\delta (\tau  )\right | \mathit{d \tau} \left | \nabla q^\delta_{\lambda +} - \nabla q^\delta \right | \dx \dt  \\
& \leq C \int \limits_0^{T- \lambda}   \int \limits_t^{ t + \lambda}  \|(|\varphi^\delta (\tau)|^3 + 1)\|_{L^3 (\Omega)}  \|\ve^\delta (\tau )\|_{L^6 (\Omega)} \mathit{d \tau} \|\nabla q^\delta_{\lambda +} - \nabla q^\delta\|_{L^2 (\Omega)}  \dt  \\
& \leq C \int \limits_0^{T-  \lambda} \lambda^\frac{1}{4} ( \|\varphi^\delta\|_{L^{12} (0,T; L^{9} (\Omega))}^3 + 1 ) \|\ve^\delta\|_{L^2 (0,T; L^6 (\Omega))} \|\nabla q^\delta_{\lambda +} - \nabla q^\delta \|_{L^2 (\Omega)} \dt \leq C (T) \lambda^\frac{1}{4} .
\end{align*}
Here we used the embeddings, see Theorem 2.32 in
\cite{Dissertation_Weber},
\begin{align*}
L^2 (0,T; H^2 (\Omega)) \cap L^\infty (0,T; L^6 (\Omega)) \hookrightarrow L^4 (0,T; L^\infty (\Omega))  , \\
L^4 (0,T; L^\infty (\Omega)) \cap L^\infty (0,T; L^6 (\Omega)) \hookrightarrow L^{12} (0,T; L^9 (\Omega)) .
\end{align*}
The second term can analogously be estimated by
\begin{align*}
& \int \limits_0^{T-\lambda} \into   \int \limits_t^{t + \lambda}  \left | g (q^\delta (\tau ) \ve^\delta (\tau) \right | \mathit{ d \tau}   \left | \nabla q^\delta_{\lambda +} - \nabla q^\delta \right | \dx \dt  \\
& \leq  \int \limits_0^{T-\lambda}  \int \limits_t^{t + \lambda} \|q^\delta (\tau) \|_{L^6 (\Omega)} \|\ve^\delta (\tau)\|_{L^3 (\Omega)} \mathit{ d \tau}  \|\nabla q^\delta_{\lambda +} - \nabla q^\delta \|_{L^2 (\Omega) } \dt  \leq C (T) \lambda^\frac{1}{4} .
\end{align*}
$\\$
Using these estimates in (\ref{delta_inequality_qNplus_s_qN_estimate}) yields that there exists a constant $C (T) > 0$ such that
 (\ref{delta_estimate_compactness_q_plus_lambda_help}) holds for every $\lambda \in (0,T)$ and therefore $(q^\delta)_{\delta > 0}$ is relatively compact in $L^2 (0,T; L^2 (\Omega))$.
Thus there exists a subsequence such that
\begin{align*}
q^\delta \rightarrow  q \qquad \text{ in } L^2 (0, T ; L^2 (\Omega)) \text{ for all }0<T<\infty 
\end{align*}
and 
\begin{align*}
q^\delta (t,x) \rightarrow q (t,x)  \qquad \text{ a.e. in } (0 , \infty) \times \Omega .
\end{align*}

Now we can show with similar arguments as in \cite{MR3084319} that it holds $\ve^\delta \rightarrow \ve$ in $L^2 (0,T; L^2 (\Omega)^d)$ as $\delta \rightarrow 0$.
To this end, we can use that $\partial_t ( \mathbb P_\sigma (\rho^\delta \ve^\delta))$ is bounded in $L^1 (0,T; H^{-2} (\Omega)^d)$ for every $0<T<\infty$ since
\begin{enumerate}
\item[i)] $\qquad \rho^\delta \ve^\delta \otimes \ve^\delta  \text{ is bounded in } L^2 (0,T; L^{\frac{3}{2}}(\Omega)^{d \times d}) $, 
\item[ii)] $\qquad \ve^\delta \otimes \Je^\delta  \text{ is bounded in } L^\frac{4}{3} (0,T; L^\frac{6}{5} (\Omega))$, 
\item[iii)] $\qquad \mu^\delta \nabla \varphi^\delta  \text{ is bounded in }  L^2 (0,T; L^{\frac{3}{2}}(\Omega)^d) $, 
\item[iv)] $\qquad \frac{h(q^\delta)}{\varepsilon} W' (\varphi^\delta) \nabla \varphi^\delta  \text{ is bounded in } L^{\frac{4}{3}} (0,T; L^{\frac{6}{5}} (\Omega)^d)$. 
\item[v)] If $\rho (\varphi) \not \equiv const.$, 
\begin{align}\label{prove_delta_rightarrow_0_tilde_R_v_delta_tested}
\left | \left \langle \frac{ R^\delta \ve^\delta}{2} , \boldsymbol \psi \right \rangle \right | &= \left | \int \limits_0^T \into  \nabla \frac{\partial \rho (\varphi^\delta)}{\partial \varphi^\delta} \cdot ( \tilde m (\varphi^\delta) \nabla \mu^\delta) \ve^\delta \cdot \boldsymbol \psi \dx \dt \right |  \leq C \|\boldsymbol \psi\|_{L^\infty (0,T; H^2 (\Omega))} 
\end{align}
for every $\boldsymbol \psi \in L^\infty(0,T; H^2 (\Omega)^d)$, where $C > 0$ is  independent of $\delta > 0$.

\end{enumerate}
We show these bounds in detail.
\begin{enumerate}
\item[Ad i)] This boundedness follows from the boundedness of
$\ve^\delta \in L^\infty (0,T; L^2_\sigma (\Omega)) \cap L^2 (0,T; L^6
(\Omega))$ and $\rho(\varphi^\delta) \in L^\infty (Q_T)$.
\item[Ad ii)] We need to estimate products of the form $\ve^\delta_k \rho' (\varphi^\delta) \tilde m (\varphi^\delta) \partial_{x_l} \mu^\delta$, where $k , l = 1,...,d$. The terms $\rho ' (\varphi^\delta)$ and $\tilde m (\varphi^\delta)$ are bounded in $L^\infty (Q_T)$. Thus the boundedness follows from the boundedness of $\ve^\delta \in L^4 (0,T; L^3 (\Omega)^d)$ and $\nabla \mu^\delta \in L^2 (Q_T)$.
\item[Ad iii)] This follows from $\mu^\delta \in L^2 (0,T; L^6 (\Omega))$ and $\nabla \varphi^\delta \in L^\infty (0,T; L^2 (\Omega)^d)$.
\item[Ad iv)] The growth conditions for $h$ and $W'$ yield the estimate
\begin{align*}
\left | \frac{h(q^\delta)}{\varepsilon} W ' (\varphi^\delta) \nabla \varphi^\delta \right | \leq \frac{C}{\varepsilon} \left ( |q^\delta| + 1\right ) \left ( |\varphi^\delta|^2 + 1 \right )  \left | \nabla \varphi^\delta \right | .
\end{align*}
By the Gagliardo-Nirenberg inequality  and the H\"older inequality we have the embedding
\begin{align*}
L^2 (0,T; H^2 (\Omega)) \cap L^\infty (0,T; L^6 (\Omega)) \hookrightarrow L^8 (0,T; L^{12} (\Omega)).
\end{align*}
Together with the boundedness of $q^\delta \in L^2 (0,T; L^6 (\Omega))$ and $ \nabla \varphi^\delta \in L^\infty (0,T; L^2 (\Omega)^d)$ we get the statement.
\item[Ad v)] We need to estimate products of the form $\rho '' (\varphi^\delta) \partial_j \varphi^\delta \tilde m (\varphi^\delta) \partial_j \mu^\delta \ve^\delta_k  \boldsymbol \psi_k$.
\\
To this end, we consider $\varphi^\delta = \varphi^\delta_1 + \varphi^\delta_2 + \varphi_0$, where $\varphi^\delta_1$ and $\varphi^\delta_2$ are the solutions of \eqref{equation_splitted_phi_phi_1}-\eqref{equation_splitted_phi_mean_value_phi_1_phi_2}.
First, we note that in (\ref{equation_splitted_phi_phi_1}), $h(q^\delta)$ is bounded in $ L^\infty (0,T; L^2 (\Omega))$ because of the growth condition \eqref{growth_condition_h}. Moreover, the growth condition $|W' (a)| \leq C ( |a|^s +1)$ for every $a \in \mathbb R$ and fixed $0 < s < 1$  implies that $W' (\varphi^\delta)$ is bounded in $L^\infty (0, T; L^{6 + s_1} (\Omega))$, where  $s_1  > 0$ depends on $s$. 
Thus $h(q^\delta) W' (\varphi^\delta)$ is bounded in $L^\infty (0,T; L^{\frac{3}{2} + s_2} (\Omega))$, where $s_2 > 0$ depends on $s$.  Due to the boundedness of $\Delta \varphi_0$ in $L^\infty (0,T; L^2 (\Omega))$, Lemma \ref{lemma_phi_solves_equation_independent_of_delta} yields that $\varphi_1^\delta $ is bounded in $L^p (0,T; W^2_{\frac{3}{2} + s_2} (\Omega))$ for every $1 \leq p < \infty$. Hence, we get that
\begin{align*}
\partial_j \varphi^\delta_1 \text{ is bounded in } L^p (0,T; W^1_{\frac{3}{2} + s_2} (\Omega)) \hookrightarrow L^p (0,T; L^{3 + s_3} (\Omega)) 
\end{align*}
for every $1 \leq p < \infty$ and $j = 1,...,d$, where $s_3> 0$ depends on $s$. 
Furthermore, Lemma~\ref{lemma_phi_solves_equation_independent_of_delta} yields that $\varphi^\delta_2$ is bounded in $L^2 (0,T; W^2_6 (\Omega))$ and therefore
\begin{align*}
\partial_j \varphi^\delta_2 \text{ is bounded in }L^2 (0,T; W^1_6 (\Omega)) \hookrightarrow L^2 (0,T; L^\infty (\Omega)) 
\end{align*}
for $j = 1,...,d$.
By interpolation we can conclude that for every $\varepsilon_2 \in (0,5]$, there exists $\varepsilon_1 > 0$ such that
\begin{align*} 
\ve^\delta   \text{ is bounded in }L^\infty (0,T; L^2_\sigma (\Omega)) \cap L^2 (0,T; L^6 (\Omega)^d) \hookrightarrow L^{2 + \varepsilon_1} (0,T; L^{6 - \varepsilon_2} (\Omega)^d) .
\end{align*}
Altogether, the boundedness of $\partial_j \mu^\delta $ in $ L^2 (0,T; L^2 (\Omega)) $, $\ve^\delta $ in $ L^{2 + \varepsilon_1} (0,T; L^{6 - \varepsilon_2} (\Omega)^d)$  and  $\partial_j \varphi^\delta_1$ in $ L^r (0,T; L^{3+ s_3} (\Omega))$ for every $1 \leq r < \infty$ and some $\varepsilon_1, \varepsilon_2, s_3 > 0$ yields
\begin{align*}
\left | \int \limits_0^T \into \rho '' (\varphi^\delta) \partial_j \varphi_1^\delta \tilde m (\varphi^\delta) \partial_j \mu^\delta \ve^\delta_k  \boldsymbol \psi_k \dx \dt \right | \leq C \|\boldsymbol \psi\|_{L^\infty (0,T; H^2 (\Omega))}
\end{align*}
for every $j,k = 1,...,d$, a constant $C > 0$ independent of $\delta$ and $\boldsymbol \psi \in L^\infty (0,T; H^2 (\Omega)^d)$. Analogously, we can conclude from $\partial_j \varphi^\delta_2 \in L^2 (0,T; L^\infty (\Omega))$, $\partial_j \mu^\delta \in L^2 (0,T; L^2 (\Omega))$ and $\ve^\delta \in L^\infty (0,T; L^2_\sigma (\Omega))$ 
\begin{align*}
\left | \int \limits_0^T \into \rho '' (\varphi^\delta) \partial_j \varphi_2^\delta \tilde m (\varphi^\delta) \partial_j \mu^\delta \ve^\delta_k  \boldsymbol \psi_k \dx \dt  \right | \leq C \|\boldsymbol \psi\|_{L^\infty (0,T; H^2 (\Omega))}
\end{align*}
for every $j,k = 1,...,d$, a constant $C > 0$ independent of $\delta$ and $\boldsymbol \psi \in L^\infty (0,T; H^2 (\Omega)^d)$. Both estimates together prove (\ref{prove_delta_rightarrow_0_tilde_R_v_delta_tested}). 
\end{enumerate}

Now the Lemma of Aubin-Lions  yields 
\begin{align*}
\mathbb P_\sigma (\rho^\delta \ve^\delta) \in L^2 (0,T; H^1_0 (\Omega)^d \cap L^2_\sigma (\Omega) )  \cap W^1_1 (0,T; V '  (\Omega)) \hookrightarrow \hookrightarrow L^2 (0,T; L^2_\sigma (\Omega)) .
\end{align*}
Thus we can show $\mathbb P_\sigma (\rho^\delta \ve^\delta ) \rightarrow \mathbb P_\sigma ( \rho \ve)$ in $L^2 (0,T; L^2_\sigma (\Omega))$ and therefore
\begin{align*}
\int \limits_0^T \into \rho^\delta |\ve^\delta|^2 \dx \dt = \int \limits_0^T \into \mathbb P_\sigma ( \rho^\delta \ve^\delta ) \cdot \ve^\delta \dx \dt \rightarrow \int \limits_0^T \into \mathbb P_\sigma (\rho \ve) \cdot \ve \dx \dt = \int \limits_0^T \into \rho |\ve|^2 \dx \dt 
\end{align*}
as $\delta \rightarrow 0$, which implies $(\rho^\delta )^\frac{1}{2} \ve^\delta \rightarrow \rho^\frac{1}{2} \ve$ in $L^2 (Q_T)$. This yields
\begin{align*}
\ve^\delta =  (\rho^\delta)^\frac{1}{2} (\rho^\delta )^\frac{1}{2} \ve^\delta  \rightarrow_{\delta\to 0} \ve \qquad \text{ in } L^2 (0,T; L^2 (\Omega)^d) 
\end{align*}
for every $T\in (0,\infty)$.
Since $\ve^\delta$ is bounded in $L^2(0,T;L^6(\Omega))$, we can even conclude $\ve^\delta \rightarrow \ve  \text{ in } L^2 (0,T; L^{6 - \varepsilon} (\Omega)^d) \text{ for every } 0 < \varepsilon \leq 5$. 

It remains to show that $(\ve, \varphi, \mu, q)$ is a weak solution of (\ref{equation_model_0}) - (\ref{equation_model_3}) in the sense of Definition \ref{definition_without_delta_weak_solutions}. To this end, we pass to the limit $\delta \rightarrow 0$ in (\ref{definition_weak_solution_time_dependent_case_0}) - (\ref{definition_weak_solution_time_dependent_case_3}). 
We prove the convergences 
\begin{align}\label{delta_to_zero_R_tilde_v_tested_with_psi}
\left \langle \frac{ R^\delta \ve^\delta}{2} , \boldsymbol \psi \right \rangle \rightarrow_{\delta\to 0} \left \langle \frac{R\ve}{2} ,\boldsymbol \psi \right \rangle
\end{align}
and
\begin{align}\label{delta_to_zero_Delta_2_ve_tested_with_psi}
\delta \int \limits_0^T \into  \Delta \ve^\delta\cdot \Delta \boldsymbol \psi \dx \dt \rightarrow_{\delta\to 0} 0
\end{align}
for all $\boldsymbol \psi \in C^\infty_0 (0,T; C^\infty_{0, \sigma} (\Omega))$ as $\delta \rightarrow 0$ in detail.
We already showed that $\left \langle \frac{ R^\delta \ve^\delta}{2} , \cdot \right \rangle$ is bounded in $L^1 (0,T; H^{-2} (\Omega)^d )$, cf.\ (\ref{prove_delta_rightarrow_0_tilde_R_v_delta_tested}).
For the proof of the convergence (\ref{delta_to_zero_R_tilde_v_tested_with_psi}), we consider $\varphi^\delta = \varphi^\delta_1 + \varphi^\delta_2 + \varphi_0$, where $\varphi^\delta_1$ and $\varphi^\delta_2$ are as before.
We have to study products of the form $ \rho '' (\varphi^\delta) \partial_j \varphi_1^\delta \tilde m (\varphi^\delta) \partial_j \mu^\delta \ve^\delta_k  \boldsymbol \psi_k$.
Since $\ve^\delta \rightarrow \ve$ in $L^2 (0,T; L^{6 - \varepsilon}_\sigma (\Omega))$ as $\delta \rightarrow 0$ for every $0 < \varepsilon \leq 5$ and since $(\ve^\delta)_{\delta > 0} , \ve$ is bounded in $L^\infty (0,T; L^2_\sigma (\Omega) ) $, it follows that for every $\varepsilon_2>0$ there is some $\varepsilon_1>0$ such that $\ve^\delta \rightarrow \ve \text{ in } L^{2 + \varepsilon_1} (0,T; L^{6 - \varepsilon_2} (\Omega)^d)$.
Due to the boundedness of $\partial_j \varphi_1^\delta$ in $ L^p (0,T; L^{3 + s_3} (\Omega))$ for every $1 \leq p < \infty$, where  $s_3 > 0$ depends on $s$, we get $\partial_j \varphi^\delta_1 \rightarrow \partial_j \varphi_1  \text{ in } L^q (0,T; L^{3 + s_4} (\Omega)) \text{ for all } 1 \leq q < \infty, \ j = 1,...,d$ as $\delta \rightarrow 0$, where $s_4\in (0, s_3)$ is arbitrary.
Altogether, we have
\begin{align*}
\partial_j \varphi^\delta_1 &\rightarrow_{\delta\to 0} \partial_j \varphi_1 && \text{ in } L^q (0,T; L^{3 + s_4} (\Omega)) \text{ for all } 1 \leq q < \infty, \ j = 1,...,d , \\
\ve^\delta &\rightarrow_{\delta\to 0} \ve && \text{ in } L^{2 + \varepsilon_1} (0,T; L^{6 - \varepsilon_2}_\sigma (\Omega)) , \\
\partial_j \mu^\delta &\rightharpoonup_{\delta\to 0} \partial_j \mu && \text{ in } L^2 (0,T; L^2 (\Omega)) \text{ for all }  j = 1,..., d.
\end{align*}
Now we choose 
 $\varepsilon_2 > 0$ so small that $\frac{1}{3 + s_4}  +  \frac{1}{6 - \varepsilon_2} + \frac{1}{2} \leq 1$. This determines $\eps_1>0$.
Then we can choose $1 \leq q < \infty$ sufficiently large such that $\frac{1}{q} + \frac{1}{2 + \varepsilon_1} + \frac{1}{2} = 1$. Therefore we can pass to the limit $\delta \rightarrow 0$ and obtain
\begin{align*}
\int \limits_0^T \into \rho '' (\varphi^\delta) \partial_j \varphi_1^\delta \tilde m (\varphi^\delta) \partial_j \mu^\delta \ve^\delta_k  \boldsymbol \psi_k \dx \dt  \rightarrow_{\delta\to 0} \int \limits_0^T \into \rho '' (\varphi) \partial_j \varphi_1 \tilde m (\varphi) \partial_j \mu \ve_k  \boldsymbol \psi_k \dx \dt
\end{align*}
for all $\boldsymbol \psi \in C^\infty_0 (0,T; C^\infty_{0, \sigma} (\Omega))$.
We also have to show the same convergence for $\varphi^\delta_2$. 
When we proved (\ref{prove_delta_rightarrow_0_tilde_R_v_delta_tested}), we already showed that $ \partial_j \varphi_2^\delta$ is bounded in $ L^2 (0,T; W^1_6 (\Omega))$. 
From $\partial_j \varphi^\delta_2 = \partial_j \varphi^\delta - \partial_j  \varphi_1^\delta - \partial_j  \varphi_0$, we can even conclude that $\partial_j \varphi^\delta_2 $ is bounded in $L^p (0,T; L^2 (\Omega))$ for every $1<p<\infty$, $0<T<\infty$ since this holds for all terms on the right-hand side. Now $\|f\|_{L^\infty(\Omega)}\leq C \|f\|_{W^1_6(\Omega)}^{\frac34}\|f\|_{L^2(\Omega)}^{\frac14}$, cf.\ \cite[Theorem 5.9]{MR2424078}, yields that
\begin{align*}
\partial_j \varphi^\delta_2 \text{ is bounded in } L^r (0,T; L^\infty (\Omega)) \text{ for every } 1\leq r<\frac{8}{3}, 0<T<\infty.
\end{align*}
Moreover, $\partial_j \varphi^\delta_2$ converges strongly in $L^2 (0,T; L^2 (\Omega))$ and almost everywhere since this is true for $\partial_j \varphi^\delta $ and $\partial_j \varphi^\delta_1 $. As a consequence, for every $1 \leq q_1 < \infty$, $1\leq r<\frac83$, and $0<T<\infty$ we have  $$\partial_j \varphi^\delta_2 \to_{\delta \to 0} \partial_j \varphi_2 \qquad \text{ in } L^r (0,T; L^{q_1} (\Omega)).$$
From the boundedness of $\ve^\delta $ in $ L^\infty (0,T; L^2_\sigma (\Omega))$ and $\ve^\delta \rightarrow \ve$ in $L^2 (0,T; L^{6 - \varepsilon} (\Omega)^d)$ for every $0 < \varepsilon \leq 5$ it follows that for every $1 \leq q_2 < \infty$ there exists $\varepsilon_2 > 0$ such that $\ve^\delta \rightarrow \ve \text{ in } L^{q_2} (0,T; L^{2 + \varepsilon_2} (\Omega)^d)$.
Thus we have
\begin{align*}
\partial_j \varphi^\delta_2 & \rightarrow_{\delta\to 0} \partial_j \varphi_2 && \text{ in } L^r (0,T; L^{q_1} (\Omega)) \text{ for all } 1 \leq q_1 < \infty , 1\leq r<\frac83,  j = 1,..., d, \\
\ve^\delta & \rightarrow_{\delta\to 0} \ve && \text{ in }  L^{q_2} (0,T; L^{2 + \varepsilon_2} (\Omega)^d) \text{ for all } 1 \leq q_2 < \infty, \\
\partial_j \mu^\delta & \rightharpoonup_{\delta\to 0} \partial_j \mu  && \text{ in } L^2 (0,T; L^2 (\Omega)) \text{ for all }  j = 1,..., d. 
\end{align*}
Now we choose $r\in (1,\tfrac83)$ and $q_2\in (2,\infty)$ such that $\frac1r+\frac12+\frac1{q_2}=1$ and  $q_1<\infty$ such that $\frac1{q_1}+\frac1{2+\varepsilon_2}+\frac12=1$. Then we obtain
\begin{align*}
\int \limits_0^T \into \rho '' (\varphi^\delta) \partial_j \varphi_2^\delta \tilde m (\varphi^\delta) \partial_j \mu^\delta \ve^\delta_k  \boldsymbol \psi_k \dx \dt  \to_{\delta\to 0} \int \limits_0^T \into \rho '' (\varphi) \partial_j \varphi_2 \tilde m (\varphi) \partial_j \mu \ve_k  \boldsymbol \psi_k \dx \dt
\end{align*}
for all $\boldsymbol \psi \in C^\infty_0 (0,T; C^\infty_{0, \sigma} (\Omega))$ as $\delta \rightarrow 0$, which shows (\ref{delta_to_zero_R_tilde_v_tested_with_psi}). It remains to prove (\ref{delta_to_zero_Delta_2_ve_tested_with_psi}). 
But this convergence follows from the energy inequality (\ref{delta_energy_estimate}), which implies the boundedness of $\delta^\frac{1}{2} \Delta \ve^\delta$ in $L^2 (0,T; L^2 (\Omega))$, and $\delta \Delta \ve^\delta = \delta ^\frac{1}{2} ( \delta^\frac{1}{2} \Delta \ve^\delta)$.

Finally, it remains to prove the energy inequality (\ref{energy_inequality}) for all $s \leq t < T$ and almost all $0 \leq s < T$ including $s=0$. With the same arguments as in \cite{MR3084319}, one can show $E_{tot} (\ve^\delta (t) , \varphi^\delta (t) , q^\delta (t)) \rightarrow E_{tot} (\ve (t), \nabla \varphi (t), q (t))$ for a.e. $t \in (0, T)$. More precisely, because of the lower semicontinuity of norms, $\varphi^\delta (x,t) \rightarrow \varphi (x,t)$, $q^\delta (x,t) \rightarrow q (x,t)$ a.e. in $\Omega\times (0,\infty)$, and
\begin{equation*}
  \left(\sqrt{\eta(\varphi^\delta)}D\ve^\delta, \sqrt{m(\varphi^\delta,q^\delta)}\nabla q^\delta  
, \sqrt{\tilde{m}(\varphi^\delta)}\nabla \mu^\delta \right)\rightharpoonup_{\delta\to 0}
\left(\sqrt{\eta(\varphi)}D\ve, \sqrt{m(\varphi,q)}\nabla q  
, \sqrt{\tilde{m}(\varphi)}\nabla \mu\right)
\end{equation*}
in $L^2(\Omega\times (0,\infty))$ it holds
\begin{align*}
\underset{\delta \rightarrow 0}{\lim \inf} \int \limits_0^T D^\delta (t) \tau (t) \dt \geq \int \limits_0^T D(t) \tau (t) \dt
\end{align*} 
for all $\tau \in W^1_1 (0,T)$ with $\tau \geq 0$ and $\tau (T) = 0$, where $D^\delta$ and $D$ are defined by
\begin{align*}
 D^\delta (t) &:= \into m (\varphi^\delta, q^\delta) \left | \nabla q^\delta \right | ^2 \dx + \into\tilde m (\varphi^\delta) |\nabla \mu^\delta|^2 \dx + \into 2 \eta (\varphi^\delta) |D\ve^\delta |^2 \dx \\
& \ \ \ \ + \delta \into |\Delta \ve^\delta|^2 \dx + \delta \into |\partial_t \varphi^\delta|^2 \dx , \\
D (t) &:=  \into m (\varphi, q) \left | \nabla q \right | ^2 \dx + \into\tilde m (\varphi) |\nabla \mu|^2 \dx + \into 2 \eta (\varphi) |D\ve|^2 \dx .
\end{align*}
From the energy estimate in the case $\delta > 0$, cf. (\ref{delta_energy_estimate}), we can conclude
\begin{align*}
E_{tot} (\ve_0 , \varphi_0, \nabla \varphi_0, q_0 ) \tau (0) + \int \limits_0^T E_{tot} (\ve^\delta (t), \varphi^\delta (t), \nabla \varphi^\delta (t), q^\delta (t) ) \tau' (t) \mathit{dt} \geq \int \limits_0^T D^\delta (t) \tau (t) \dt .
\end{align*}
Therefore, it follows in the limit $\delta \rightarrow 0$
\begin{align*}
E_{tot} (\ve_0 , \varphi_0, \nabla \varphi_0, q_0 ) \tau (0) + \int \limits_0^T E_{tot} (\ve (t), \varphi (t), \nabla \varphi (t), q (t) ) \tau' (t) \mathit{dt} \geq \int \limits_0^T D (t) \tau (t) \dt 
\end{align*}
for all $\tau \in W^1_1 (0,T)$  with $\tau \geq 0$ and $\tau (T) = 0$. But this implies the energy inequality because of \cite[Lemma 4.3]{MR2504845}.
\end{proof}


\section{Existence for the Approximate System}\label{Appendix}
It is the aim of this section to prove the existence of weak solutions for the approximating system \eqref{equation_delta_model_0}-\eqref{equation_delta_model_3}. To this end, we determine an appropriate time discretization. We solve the time-discrete problem by using the Leray-Schauder principle, cf.\ Theorem \ref{lemma_existence_solutions_time_discrete_problem} below. Then we prove Theorem \ref{theorem_existence_of_weak_solutions}.

For the time discretization, we set $h = \frac{1}{N}$ for $N \in \mathbb N$. Moreover, let $\ve_k \in L^2_\sigma (\Omega)$, $\varphi_k \in H^2_n (\Omega) $ and $q_k \in L^2 (\Omega)$ be given.
We determine $ (\ve_{k+1}, \varphi_{k+1} , \mu_{k+1}, q_{k+1} )$ as a weak solution of the system
\begin{align}
0 &= - \frac{ \rho_{k+1} \ve_{k+1} - \rho_k \ve_k}{h} - \di ( \rho_k \ve_{k+1} \otimes \ve_{k+1}) - \di ( \ve_{k+1} \otimes \Je_{k+1})+ \di \left ( 2 \eta (\varphi_{k}) D\ve_{k+1} \right ) \nonumber  \\
& \ \ \  - \nabla p_{k+1} + \frac{\tilde R_{k+1} \ve_{k+1}}{2}  + \left ( \mu_{k+1} - \frac{h(q_{k+1})}{\varepsilon} W' (\varphi_{k }) \right ) \nabla \varphi_{k } - \delta \Delta^2 \ve_{k+1}  \label{equation_time_discrete_0} ,  \\
0&=\di (\ve_{k+1}) , \\
0 & = \frac{1}{\varepsilon} \left ( \frac{f(q_{k+1}) - f(q_k)}{h} W (\varphi_k) + f(q_{k+1}) \frac{ W (\varphi_{k+1} ) - W (\varphi_k)}{h} \right ) \nonumber \\
& \ \ \ + \frac{g(q_{k+1}) - g( q_k)}{h} + \nabla \left ( \frac{1}{\varepsilon} f(q_{k+1}) W(\varphi_{k}) + g(q_{k+1}) \right ) \cdot \ve_{k+1} \nonumber \\
& \ \ \  - \di \left ( m(\varphi_k , q_k) \nabla  q_{k+1} \right ) \label{equation_time_discrete_1}    , \\
0 & = \frac{\varphi_{k+1} - \varphi_k}{h} + \nabla \varphi_k \cdot \ve_{k+1} - \di ( \tilde m (\varphi_k) \nabla \mu_{k+1} ) \label{equation_time_discrete_2} , \\
\mu _{k+1} & = - \varepsilon \Delta \varphi_{k+1} + h (q_{k+1}) \frac{1}{\varepsilon} H(\varphi_{k+1} , \varphi_k ) +  \delta \frac{\varphi_{k+1} - \varphi_k}{h} , \label{equation_time_discrete_3} 
\end{align}
with boundary conditions
\begin{align}\label{time_discrete_problem_boundary_values}
  \ve_{k+1}|_{\partial \Omega} = \Delta \ve_{k+1}|_{ \partial \Omega}= \partial_n \varphi_{k+1}|_{\partial \Omega} =  \partial_n \mu_{k+1}|_{\partial \Omega} =  \partial_n q_{k+1}|_{\partial \Omega} = 0 ,
\end{align}
where 
\begin{align}
\Je_{k+1} &= - \frac{\partial \rho }{\partial \varphi}(\varphi_k) \tilde m (\varphi_k) \nabla \mu_{k+1} ,    \label{definition_tilde_J_k_plus_1} \\
\tilde R_{k+1} &= \frac{ \rho (\varphi_{k+1}) - \rho (\varphi_k)}{h} + \di ( \rho (\varphi_k) \bold v_{k+1} + \Je_{k+1} )   \label{definition_R_k_plus_1}
\end{align}
and
$H : \mathbb R \times \mathbb R \rightarrow \mathbb R$ is defined by
\begin{align*}
H(a,b) := \begin{cases}
\frac{W(a) - W(b)}{a-b} & \text{ if } a \neq b , \\
W' (b) & \text{ if } a = b .
\end{cases} 
\end{align*}
Note that $H(a,b) (a-b) = W(a) - W(b)$ for every $a, b \in \mathbb R$.
It remains to define a weak solution for the time-discrete problem  \eqref{equation_time_discrete_0}-\eqref{time_discrete_problem_boundary_values}.

\begin{definition}[Weak solution of the time-discrete problem]~\label{definition_time_discrete_weak_solution}
\\
We call 
\begin{align*}
( \bold v_{k+1}, \varphi_{k+1}, \mu_{k+1}, q_{k+1})  \in  V(\Omega) \times H^2_n (\Omega) \times H^2_n (\Omega) \times H^1 (\Omega)
\end{align*}
 a weak solution of \eqref{equation_time_discrete_0}-\eqref{time_discrete_problem_boundary_values} for initial data $\bold v_k \in L^2_\sigma (\Omega)$, $ \varphi_k \in H^2_n (\Omega)$ and $q_k \in L^2 (\Omega)$ if it holds
\begin{align}\label{time_discrete_equation_test_version_0}
&\into  \frac{ \rho_{k+1} \bold v_{k+1} - \rho_k \bold v_k}{h} \cdot \boldsymbol \psi \dx  + \into  \mathrm{div} (\rho_k \bold v_{k+1} \otimes \bold v_{k+1} ) \cdot  \boldsymbol \psi \dx + \into 2 \eta (\varphi_k) D \bold v_{k+1} : D \boldsymbol \psi \dx \nonumber \\
& \  \ \  -  \into  (  \Je_{k+1} \otimes \bold v_{k+1}  ) : \nabla \boldsymbol \psi \dx - \frac12\left \langle \tilde R_{k+1} \bold v_{k+1} , \boldsymbol \psi \right \rangle + \delta \into \Delta \bold v_{k+1} \cdot \Delta \boldsymbol \psi \dx \nonumber \\
& = \into \left ( \mu_{k+1} - \frac{h(q_{k+1})}{\varepsilon} W' (\varphi_k) \right ) \nabla \varphi_k \cdot \boldsymbol \psi \dx
\end{align} 
for all $ \boldsymbol \psi \in V (\Omega)$, where $\Je_{k+1}$ is defined as in \eqref{definition_tilde_J_k_plus_1}, and
\begin{align}\label{time_discrete_equation_test_version_1}
\into &  \left ( \frac{1}{\varepsilon}   f(q_{k+1}) W(\varphi_{k}) + g(q_{k+1}) \right  ) \bold v_{k+1}  \cdot \nabla \phi \dx  = \into m(\varphi_k , q_k) \nabla q_{k+1} \cdot \nabla \phi \dx \nonumber \\
& + \frac{1}{ h} \into \left ( \frac{ f(q_{k+1}) W(\varphi_{k+1})}{\varepsilon} + g(q_{k+1}) - \frac{f(q_k) W(\varphi_k)}{\varepsilon} - g(q_k) \right ) \phi \dx ,
\end{align}
\begin{align}\label{time_discrete_equation_test_version_2}
0 &= \into \tilde m (\varphi_k) \nabla \mu_{k+1} \cdot \nabla \phi \dx + \into \frac{\varphi_{k+1} - \varphi_k}{h} \phi \dx + \into ( \nabla \varphi_k \cdot \bold v_{k+1} ) \phi \dx ,  
\end{align}
\begin{align}\label{time_discrete_equation_test_version_3}
\into \mu_{k+1} \phi \dx =& \into \varepsilon \nabla \varphi_{k+1} \cdot \nabla \phi \dx + \into h(q_{k+1} ) \frac{1}{\varepsilon} H(\varphi_{k+1}, \varphi_k ) \phi \dx  \nonumber \\
& + \delta \into \frac{\varphi_{k+1} - \varphi_k}{h} \phi \dx
\end{align}
for all $\phi \in H^1 (\Omega)$, where we define for $ \boldsymbol \psi \in V (\Omega)$
\begin{align*}
 \left \langle \tilde R_{k+1} \bold v_{k+1} , \boldsymbol \psi \right \rangle := \into \frac{\rho_{k+1} - \rho_k}{h} \bold v_{k+1} \cdot \boldsymbol \psi \dx - \into \left ( \rho_k \bold v_{k+1} + \Je_{k+1} \right ) \cdot \nabla ( \bold v_{k+1} \cdot \boldsymbol \psi ) \dx .
\end{align*}
\end{definition}

Note that \eqref{time_discrete_equation_test_version_0} can equivalently be written as
\begin{align}\label{reformulation_of_equation_time_discrete_0_using_div_v_J}
 & \into \left ( \frac{ \rho_{k+1} \ve_{k+1}  - \rho_k \ve_k}{h}  + \di ( \rho_k \ve_{k+1} \otimes \ve_{k+1}) \right ) \cdot \boldsymbol \psi \dx + \into  2 \eta (\varphi_{k}) D\ve_{k+1} : D \boldsymbol \psi \dx  \nonumber \\
& \ \   + \into  \left ( \di \Je_{k+1} - \frac{\rho_{k+1} - \rho_k}{h} - \ve_{k+1} \cdot \nabla \rho_k \right ) \frac{\ve_{k+1}}{2}  \cdot \boldsymbol \psi \dx + \into ( \Je_{k+1} \cdot \nabla ) \ve_{k+1} \cdot \boldsymbol \psi \dx  \nonumber \\
& \ \ + \delta \into \Delta \ve_{k+1} \Delta \boldsymbol \psi \dx  = \into \left ( \mu_{k+1} - \frac{h(q_{k+1})}{\varepsilon} W' (\varphi_{k }) \right ) \nabla \varphi_{k } \cdot \boldsymbol \psi \dx 
\end{align}
by using $\di ( \ve_{k+1} \otimes \Je_{k+1} ) = \di ( \Je_{k+1} ) \ve_{k+1} + ( \Je_{k+1} \cdot \nabla ) \ve_{k+1}$.

\begin{theorem}[Existence of weak solutions for the time-discrete problem]\label{lemma_existence_solutions_time_discrete_problem}~ \\
Let $\bold v_k \in  L^2_\sigma (\Omega)$, $\varphi_k \in H^2_n (\Omega)$ and  $q_k \in L^2 (\Omega)$ be given. Then there exist $\bold v_{k+1} \in V (\Omega)$, $\varphi_{k+1} \in H^2_n (\Omega)$, $ \mu_{k+1} \in H^2_n (\Omega)$ and $ q_{k+1} \in H^1 (\Omega)$ solving \eqref{equation_time_discrete_0}-\eqref{time_discrete_problem_boundary_values} in the sense of  Definition~\ref{definition_time_discrete_weak_solution}. Moreover,  the discrete energy estimate
\begin{align}\label{equation_discrete_energy_estimate}
& E_{tot}  (\bold v_{k+1}, \varphi_{k+1}, q_{k+1}) + \into \frac{\rho_k |\bold v_{k+1} - \bold v_k|^2}{2} \dx + h \into 2 \eta (\varphi_{k}) |D \bold v_{k+1}|^2 \dx \nonumber \\
& +  h \into m(\varphi_{k}, q_{k}) |\nabla  q_{k+1}|^2 \dx + h \into \tilde m (\varphi_k) |\nabla \mu_{k+1}|^2 \dx + \varepsilon \into \frac{ |\nabla \varphi_{k+1} - \nabla \varphi_k|^2}{2} \dx \nonumber \\ 
& + \delta \into \frac{|\varphi_{k+1} - \varphi_k|^2}{h} \dx + \delta h \into |\Delta \bold v_{k+1} |^2 \dx  \leq E_{tot}(\bold v_k, \varphi_k, q_k)  
\end{align}
is satisfied.
\end{theorem}
The proof generalizes  the proof of \cite[Lemma 4.2]{MR3084319}, where the existence of weak solutions for the model without surfactants developed in \cite{MR2890451} is shown.

\begin{proof}[Proof of Theorem \ref{lemma_existence_solutions_time_discrete_problem}]~\\
We start with the proof that the energy estimate (\ref{equation_discrete_energy_estimate}) holds for any weak solution of \eqref{equation_time_discrete_0}-\eqref{equation_time_discrete_3} in the sense of Definition \ref{definition_time_discrete_weak_solution}.
To this end, we test equation (\ref{time_discrete_equation_test_version_0}) with $\ve_{k+1}$ (i.e., choose $\boldsymbol{\psi}=\ve_{k+1}$), (\ref{time_discrete_equation_test_version_1}) with $q_{k+1}$, (\ref{time_discrete_equation_test_version_2}) with $\mu_{k+1}$ and (\ref{time_discrete_equation_test_version_3}) with $\frac{\varphi_{k+1} - \varphi_k}{h}$.
Then one important tool to simplify the equations is the identity
\begin{align*}
\bold a \cdot ( \bold a - \bold b) = \frac{| \bold a |^2}{2} - \frac{|\bold b|^2}{2} + \frac{|\bold a - \bold b|^2}{2} \qquad \text{ for } \bold a, \bold b \in \mathbb R^d .
\end{align*}
Note that when we test (\ref{time_discrete_equation_test_version_3}) with  $\frac{\varphi_{k+1} - \varphi_k}{h}$,  we use the identity $H (\varphi_{k+1} , \varphi_k) (\varphi_{k+1} - \varphi_k) = W (\varphi_{k+1} ) - W (\varphi_k)$. In \eqref{time_discrete_equation_test_version_1} tested with $q_{k+1}$, we use $f(q)  = - h' (q)$ for all $q \in \mathbb R$ and the fact that $h$ is a concave function to obtain
\begin{align*}
( f (q_{k+1} ) - f(q_k)) q_{k+1} &= f (q_{k+1} ) q_{k+1}  - f(q_k) q_k + f(q_k) (q_k - q_{k+1} ) \\
&= f (q_{k+1} ) q_{k+1}  - f(q_k) q_k + h' (q_k) ( q_{k+1} - q_k ) \\
& \geq f (q_{k+1} ) q_{k+1}  - f(q_k) q_k +  h (q_{k+1} ) - h(q_k).
\end{align*}
Moreover, we can estimate
\begin{align*}
(g (q_{k+1}) - g (q_k)) q_{k+1} = \int_{q_k}^{q_{k+1}}(-g'(s))ds\, q_{k+1}\geq -\int_{q_k}^{q_{k+1}}g'(s)s\, ds \geq G(q_{k+1}) - G(q_k) 
\end{align*}
and have $d(q) = h(q) + f(q)q$ for all $q \in \mathbb R$. Altogether this  yields  \eqref{equation_discrete_energy_estimate}.

Now we need to prove the existence of weak solutions for the time-discrete problem \eqref{equation_time_discrete_0}-\eqref{equation_time_discrete_3}. To this end, we define two operators $\mathcal L_k , \mathcal F_k : X \rightarrow Y$ and apply the Leray-Schauder principle, where
\begin{align*}
X &:= V(\Omega) \times H^1 (\Omega) \times H^2_n (\Omega) \times H^1 (\Omega) ,\\
Y &:= V' (\Omega) \times H^{-1}_0 (\Omega) \times L^2 (\Omega) \times H^{-1}_0 (\Omega).
\end{align*}
Here $V ' (\Omega):= V (\Omega)'$ and $H^{-1}_0 (\Omega) := (H^1 (\Omega))'$.
For $\bold w_{k+1} := (\ve_{k+1},q_{k+1}, \mu_{k+1}, \varphi_{k+1}) \in X$ we define the operator $\mathcal L_k : X \rightarrow Y$ by
\begin{align*}
\mathcal L_k (\bold w_{k+1}) =  \begin{pmatrix} \mathcal A (\varphi_k) \ve_{k+1} \\  \di_N ( m(\varphi_{k}, q_{k}) \nabla q_{k+1}) - \into q_{k+1} \dx \\ \di ( \tilde m (\varphi_k) \nabla \mu_{k+1})  - \into \mu_{k+1} \dx \\  \varepsilon \Delta_N \varphi_{k+1} - \into \varphi_{k+1} \dx \end{pmatrix} ,
\end{align*}
where $\mathcal A (\varphi_k) : V (\Omega) \rightarrow V' ( \Omega)$ is given by
\begin{align*}
\left \langle \mathcal A (\varphi_k) \ve_{k+1} ,  \boldsymbol \psi \right \rangle := - \into 2 \eta (\varphi_k)  D \ve_{k+1} : D \boldsymbol \psi \dx - \delta \into \Delta \ve_{k+1} \Delta \boldsymbol \psi \dx
\end{align*}
for all $\boldsymbol \psi  \in V (\Omega)$ and $\di_N : L^2 (\Omega)^d \rightarrow H^{-1}_0  (\Omega) $ and $\Delta_N : H^1 (\Omega) \rightarrow H^{-1}_0 (\Omega)$ are defined by
\begin{align*}
\left \langle \di_N \ \bold f , \phi \right \rangle := - \into \bold f \cdot \nabla \phi \dx , \qquad
\left \langle \Delta_N \varphi, \phi \right \rangle := - \into \nabla \varphi \cdot \nabla \phi \dx
\end{align*}
for all $\bold f \in L^2 (\Omega)^d$, $\varphi \in H^1 (\Omega)$ and $\phi  \in H^1 (\Omega)$.
Moreover, we define for \linebreak $ \bold w_{k+1} = (\ve_{k+1}, q_{k+1}, \mu_{k+1}, \varphi_{k+1}) \in X$ the operator $\mathcal F_k : X \rightarrow Y$ by
\begin{align*}
\mathcal F_k (\bold w_{k+1}) = \begin{pmatrix}\frac{\rho_{k+1} \ve_{k+1} - \rho_k \ve_k}{h}  + \left ( \di \Je_{k+1} - \frac{\rho_{k+1} - \rho_k}{h} - \ve_{k+1} \cdot \nabla \rho_k \right ) \frac{\ve_{k+1}}{2}  \\ 
+ ( \Je_{k+1} \cdot \nabla ) \ve_{k+1}  - \left ( \mu_{k+1} - \frac{h(q_{k+1})}{\varepsilon} W'(\varphi_k) \right ) \nabla \varphi_k \\
+ \di (\rho_k \ve_{k+1} \otimes \ve_{k+1})  \\ 
\ \\
\frac{1}{\varepsilon} \left ( \frac{f(q_{k+1}) - f(q_k)}{h} W(\varphi_k) + f(q_{k+1}) \frac{W(\varphi_{k+1}) - W(\varphi_k)}{h} \right ) + \frac{g(q_{k+1}) - g( q_k)}{h} \\ 
 + \nabla \left ( \frac{1}{\varepsilon} f(q_{k+1}) W(\varphi_k) + g(q_{k+1}) \right ) \cdot \ve_{k+1} -  \into q_{k+1} \dx \\ 
\ \\
\frac{\varphi_{k+1} - \varphi_k}{h} + \nabla \varphi_k \cdot \ve_{k+1} - \into \mu_{k+1} \dx \\
\ \\  
 h(q_{k+1}) \frac{1}{\varepsilon} H(\varphi_{k+1} , \varphi_k)  - \mu_{k+1} + \delta \frac{\varphi_{k+1} - \varphi_k}{h} - \into \varphi_{k+1} \dx \end{pmatrix} ,
\end{align*}
where $\Je_{k+1}$ is given as in (\ref{definition_tilde_J_k_plus_1}). Note that for the definition of $\mathcal F_k : X \rightarrow Y$ we used \eqref{reformulation_of_equation_time_discrete_0_using_div_v_J}, which is equivalent to \eqref{time_discrete_equation_test_version_0}.
Then  it holds 
\begin{align*}
\mathcal L_k (\bold w_{k+1}) - \mathcal F_k (\bold w_{k+1}) = 0 \qquad \text{ in } Y
\end{align*}
if and only if $\bold w_{k+1} = (\ve_{k+1}, q_{k+1}, \mu_{k+1}, \varphi_{k+1}) \in X$ is a weak solution of \eqref{equation_time_discrete_0}-\eqref{equation_time_discrete_3}. 
For the equation 
\begin{align*}
\mathcal L (\bold w_{k+1} ) = \bold F  
\end{align*}
with $\bold F \in Y$, the Lax-Milgram theorem yields the existence of unique weak solutions $\ve_{k+1} \in V (\Omega)$ and $\varphi_{k+1}$, $\mu_{k+1}$, $q_{k+1} \in H^1 (\Omega)$. 
By elliptic regularity we obtain that  $\mu_{k+1}\in H^2_n (\Omega)$.
Hence $\mathcal L_k : X \rightarrow Y$ is invertible and the open mapping theorem yields the boundedness of  $\mathcal L_k ^{-1} : Y \rightarrow X$.
\\
The next step is to show that $\mathcal F_k : X \rightarrow  Y$ is a compact operator. 
To this end, we introduce the Banach space
\begin{align*}
\tilde Y := L^{\frac{4}{3}}_\sigma  (\Omega) \times L^{\frac{4}{3}} (\Omega) \times W^1_\frac{3}{2} (\Omega) \times L^2 (\Omega) .
\end{align*}
We can conclude that $\mathcal F_k : X \rightarrow \tilde Y$ is bounded since it holds for $\bold w_{k+1} \in X$ 
\begin{align*}
\|\di (\rho_k \ve_{k+1} \otimes \ve_{k+1})\|_{L^{\frac{4}{3}} (\Omega)} & \leq C_k  \|\ve_{k+1}\|^2_{H^1 (\Omega)},  \\
\|(\di \Je_{k+1}) \ve_{k+1}\|_{L^\frac{4}{3} (\Omega)} & \leq C_k \|\ve_{k+1}\|_{H^1 (\Omega)} \|\mu_{k+1}\|_{H^2 (\Omega)}  ,  \\
\norm{h(q_{k+1}) W' (\varphi_k) \nabla \varphi_k}_{L^{\frac{4}{3}} (\Omega)} & \leq C_k \left ( \|q_{k+1}\|_{H^1 (\Omega)} + 1 \right ) , \\
\norm{\nabla ( f(q_{k+1}) W(\varphi_{k}) ) \cdot \ve_{k+1}}_{L^{\frac{4}{3}}(\Omega)} & \leq C_k  \|q_{k+1}\|_{H^1 (\Omega)}  \|\ve_{k+1}\|_{H^1 (\Omega)} , \\
\norm{\nabla g(q_{k+1}) \cdot \ve_{k+1}}_{L^{\frac{4}{3}} (\Omega)} & \leq C \|q_{k+1}\|_{H^1 (\Omega)} \|\ve_{k+1}\|_{H^1 (\Omega)} , \\
\|h(q_{k+1}) H(\varphi_{k+1}, \varphi_k) \|_{L^2 (\Omega)} & \leq C_k \left ( \|q_{k+1}\|_{L^6 (\Omega)} + 1 \right ) \left ( \|\varphi_{k+1}\|^2_{L^6 (\Omega)} + 1 \right ) .
\end{align*}
These estimates follow from Sobolev embeddings, the boundedness of $f$ together with the growth conditions $|h (q)| \leq C (|q|+ 1)$, $|W(q)| \leq C (|q|^3 +1)$, $|W' (q)| \leq C (|q|^2 + 1)$ and $|G' (q)| \leq C (|q|+ 1)$ for all $q \in \mathbb R$ and a constant $C > 0$. Moreover, we use the identities $f = - h'$, $G' (q) = g'(q)q$ and the fact that $f$ is constant outside an interval $[q_{min} , q_{max}]$. 
\\
The continuity of the linear terms of $\mathcal F_k : X \rightarrow \tilde Y$ follows from their boundedness. For the nonlinear terms, the continuity can be proven with the aid of the theorem on continuity of Nemyckii operators on $L^p$-spaces and with the multilinear structures. Since $\mathcal F_k : X \rightarrow \tilde Y$ is a continuous and bounded operator and since the embeddings $H^1 (\Omega) \hookrightarrow  L^4 (\Omega)$, $L^{\frac43}(\Omega)\hookrightarrow H^{-1}_0(\Omega)$, and $W^1_{\frac32} (\Omega) \hookrightarrow  L^2 (\Omega)$ are compact, we can conclude that $\mathcal F_k : X \rightarrow Y $ is a compact operator.
\\
In the following, we want to apply the Leray-Schauder principle.
 We already noted that $\bold w_{k+1} \in X$ is a weak solution of   (\ref{equation_time_discrete_0}) - (\ref{equation_time_discrete_3}) if and only if  $\mathcal L _k (\bold w_{k+1}) - \mathcal F_k ( \bold w_{k+1}) = 0 \text{ in } Y $.
This is equivalent to
\begin{align*}
\bold g_{k+1} - \mathcal F_k \circ \mathcal L_k ^{-1} (\bold g_{k+1}) = 0 \qquad \text{ in } Y \quad \text{ for } \bold g_{k+1} := \mathcal L_k (\bold w_{k+1}) .
\end{align*}
We set $\mathcal K _k := \mathcal F_k \circ \mathcal L_k^{-1} :  Y \rightarrow Y$ and note that proving the existence of a weak solution for  (\ref{equation_time_discrete_0}) - (\ref{equation_time_discrete_3}) is equivalent to proving the existence of a fixed-point of
\begin{align*}
\bold g_{k+1} - \mathcal K_k (\bold g_{k+1}) = 0 \quad \text{ in } Y \quad \Leftrightarrow \quad \bold g_{k+1} = \mathcal K_k ( \bold g_{k+1}) \quad \text{ in } Y.
\end{align*}
We can prove the existence of such a fixed-point with the Leray-Schauder principle, cf. \cite[Theorem 1.D.]{MR1347691}. We have to show:
\begin{align}\label{property_leray_schauder_principle}
& \text{There exists }  r_{k+1}  > 0  \text{ such}  \text{ that} \text{ if }  \bold g_{k+1} \in  Y    \text{ solves } \bold g_{k+1} = \lambda \mathcal K_k \bold g_{k+1}  \nonumber \\
& \text{with } 0 \leq \lambda < 1   , \text{ then it holds  } \|\bold g_{k+1}\|_{Y} \leq r_{k+1}.
\end{align}
To this end, let $\bold g_{k+1} \in  Y$ and $0 \leq \lambda < 1$ be given such that $\bold g_{k+1} = \lambda \mathcal K_k \bold g_{k+1}$. As $ \bold w_{k+1} = \mathcal L_k^{-1} (\bold g_{k+1}) \in X$, we conclude
\begin{align*}
\bold g_{k+1} = \lambda  \mathcal K_k ( \bold g_{k+1}) \text{ in } Y \quad \Leftrightarrow \quad \mathcal L_k ( \bold w_{k+1}) - \lambda \mathcal F_k (\bold w_{k+1}) = 0 \text{ in } Y.
\end{align*}
Testing (\ref{time_discrete_equation_test_version_0}) with $\ve_{k+1}$, (\ref{time_discrete_equation_test_version_1}) with $q_{k+1}$, (\ref{time_discrete_equation_test_version_2}) with $\mu_{k+1}$ and (\ref{time_discrete_equation_test_version_3}) with $\frac{\varphi_{k+1} - \varphi_k}{h}$, using the same identities as in the derivation of the energy inequality and omitting some non-negative terms, we can estimate $\bold w_{k+1} = \mathcal L_k^{-1} (\bold g_{k+1}) = (\ve_{k+1}, \varphi_{k+1} , \mu_{k+1}, q_{k+1})$ in  $X$ by
\begin{align*}
\|\ve_{k+1}\|^2_{H^2 (\Omega)} +  \|q_{k+1}\|^2_{H^1 (\Omega)} +\|\mu_{k+1}\|^2_{H^2 (\Omega)} + \|\varphi_{k+1}\|^2_{H^1 (\Omega)} \leq C_k .
\end{align*}
For more details, we refer to \cite[Section~3.2.3]{Dissertation_Weber}.
Thus (\ref{property_leray_schauder_principle}) is fulfilled and the Leray-Schauder principle yields the existence of $\bold g_{k+1} \in  Y$ such that $ \bold g_{k+1} - \mathcal K_k (\bold g_{k+1}) = 0$, which is equivalent to $\mathcal L_k (\bold w_{k+1}) - \mathcal F_k (\bold w_{k+1}) = 0$, where $\bold w_{k+1} = \mathcal L_k^{-1} (\bold g_{k+1})$.

Finally, we need to show higher regularity for $\varphi_{k+1}$. From $\mathcal L _k (\bold w_{k+1}) = \mathcal F_k (\bold w_{k+1})$ with $\bold w_{k+1} = (\ve_{k+1}, q_{k+1} , \mu_{k+1} , \varphi_{k+1})$ it follows
\begin{align*}
\varepsilon \Delta_N \varphi_{k+1} = - \mu_{k+1} + h(q_{k+1}) \frac{1}{\varepsilon}H (\varphi_{k+1}, \varphi_k) + \delta \frac{\varphi_{k+1} - \varphi_k}{h} \qquad \text{ in } H^{-1}_0 (\Omega),
\end{align*} 
where the right-hand side is bounded in the $L^2$-norm. Thus elliptic regularity theory yields $\varphi_{k+1} \in H^2_n (\Omega)$.
Hence, there exists a weak solution for the time-discrete problem \eqref{equation_time_discrete_0}-\eqref{equation_time_discrete_3} in the sense of Definition \ref{definition_time_discrete_weak_solution}, which fulfills the discrete energy estimate (\ref{equation_discrete_energy_estimate}).
\end{proof}

Using Theorem \ref{lemma_existence_solutions_time_discrete_problem}, we can prove the existence of weak solutions for the approximating system \eqref{equation_delta_model_0}-\eqref{equation_delta_model_3}.
\begin{proof}[Proof of Theorem \ref{theorem_existence_of_weak_solutions}]~ \\
We start with fixed $ N \in \mathbb N$ and set $h = \frac{1}{N}$.
Then Theorem \ref{lemma_existence_solutions_time_discrete_problem} iteratively yields the existence of weak solutions
\begin{align*}
(\ve_{k+1}, q_{k+1}, \mu_{k+1} , \varphi_{k+1} )  \in V(\Omega) \times  H^1 (\Omega) \times H^2_n(\Omega) \times H^2_n (\Omega) 
\end{align*}
for the time-discrete problem \eqref{equation_time_discrete_0}-\eqref{equation_time_discrete_3}. 
We define interpolating functions $f^N (t)$ on $[- h , \infty)$ by $f^N (t) = f_k$ for  $ t \in [(k-1)h, kh)$, where $k \in \mathbb N_0$ and $f_k \in \{ \ve_k,  \varphi_{k},  q_{k} \}$, resp. $\mu^N (t)$ on $[0, \infty)$ by $\mu^N (t) = \mu_k$ for $ t \in [(k-1)h, kh)$, where $k \in \mathbb N$. 
With these definitions it holds $f^N ( (k-1)h ) = f_k$, $f^N (kh) = f_{k+1}$ and $f^N (t) = f_{k+1}$ for $t \in [kh, (k+1)h )$ for $f^N \in \{\ve^N, \varphi^N, q^N\}$, $k \in \mathbb N_0$, and $\mu^N ((k-1)h) = \mu_k$,  $\mu^N (kh) = \mu_{k+1}$ for $k \in \mathbb N$.
Furthermore, we use the abbreviations
\begin{align*}
( \Delta _h^+ f) (t) &:= f (t+h) - f(t) , \quad  &&(\Delta _h^- f) (t) := f(t) - f(t-h) , \\
\partial_{t,h}^+ f(t) &:= \frac{1}{h} (\Delta^+ _h f) (t) , \quad && \partial_{t,h}^- f (t) := \frac{1}{h} (\Delta _h^- f) (t) , \\
f_h (t) &:= (\tau_h^* f) (t) = f(t-h) ,  && f_{h+} (t) := f(t+h) 
\end{align*}
and set
\begin{align*}
\rho^N &:= \rho (\varphi^N)  , \qquad \rho^N_h := \rho (\varphi^N_h) , \qquad \Je^N := - \rho ' (\varphi^N_h) \tilde m (\varphi^N_h) \nabla \mu^N , \\
\tilde R^N &:= \partial^-_{t,h} \rho^N + \di \left ( \rho^N_h \ve^N + \Je^N \right ) .
\end{align*}
We choose an arbitrary $\boldsymbol \psi \in C^{\infty}_{(0)} ([0, \infty) ; V (\Omega) )$ and set $\boldsymbol{ \tilde  \psi_k }:= \int \limits_{kh}^{(k+1)h} \boldsymbol \psi \dt$ as test function in (\ref{time_discrete_equation_test_version_0}). Then we sum over $k \in \mathbb N_0$. This yields
\begin{align}\label{equation_discrete_continuous_0}
& - \int \limits_0^{\infty}  \into \rho^N \ve^N  \cdot \partial^+_{t,h} \boldsymbol \psi \dx \dt-\into \rho(\phi_0)\ve_0\cdot\frac1h \int_0^h\boldsymbol \psi\dt \dx + \int \limits_0^\infty \into  ( \rho^N_h \ve^N \otimes \ve^N ) : \nabla \boldsymbol \psi \dx \dt \nonumber \\
& +  \int \limits_0^\infty \into 2 \eta (\varphi^N_h ) D \ve^N : D \boldsymbol \psi \dx \dt     - \int \limits_0^\infty \into \left ( \Je^N \otimes \ve^N  \right ) : \nabla \boldsymbol \psi \dx \dt - \left \langle \frac{\tilde R^N \ve^N}{2} , \boldsymbol \psi \right \rangle  \nonumber \\
&  + \delta \int \limits_0^\infty \into  \Delta \ve\cdot \Delta \boldsymbol \psi \dx \dt =  \int \limits_0^\infty \into \left ( \mu^N - \frac{h(q^N)}{\varepsilon} W' (\varphi^N_h ) \right ) \nabla \varphi^N_h \cdot \boldsymbol \psi \dx \dt  
\end{align}
for all $\boldsymbol \psi \in C^\infty _{(0)} ( [0, \infty) ; V (\Omega) )$, where $ \left \langle \frac{\tilde R^N \ve^N}{2} , \boldsymbol \psi \right \rangle$ is defined analogously to (\ref{definition_Rv_2_psi}):
\begin{align*}
 \left \langle \frac{\tilde R^N \ve^N}{2} , \boldsymbol \psi \right \rangle :=& \frac{1}{2} \int \limits_0^\infty \into \frac{\rho^N - \rho^N_h}{h} \ve^N \cdot \boldsymbol \psi \dx \dt  
 -  \frac{1}{2}   \int \limits_0^\infty \into \left (  \rho^N_h \ve^N  + \Je^N \right ) \cdot \nabla ( \ve^N  \cdot \boldsymbol \psi )   \dx \dt .
\end{align*}
Now let $\phi \in C^{\infty}_{(0)}([0, \infty); C^1 (\overline \Omega ))$ be arbitrary. We set  $\tilde \phi := \int \limits_{kh}^{(k+1)h} \phi \dt$ as test function in \eqref{time_discrete_equation_test_version_1}-\eqref{time_discrete_equation_test_version_3} and sum over $k \in \mathbb N_0$ again. Then we get
\begin{alignat}{1}
 &-\int \limits_0^{\infty} \into  \left ( f(q^N) W(\varphi^N) + g(q^N ) \right ) \partial^+_{t,h}\phi \dx \dt - \into (f(q_0) W(\varphi_0) + g(q_0))\frac1h \into \phi\dt dx \label{equation_discrete_continuous_1}   \\\nonumber
& \ \  - \int \limits_0^\infty \into   \left ( \frac{1}{\varepsilon} f(q^N) W(\varphi^N_h) + g(q^N) \right ) \ve^N \cdot \nabla \phi \dx \dt  = - \int \limits_0^{\infty} \into m(\varphi^N_h , q^N_h ) \nabla q^N  \cdot \nabla \phi \dx \dt 
\end{alignat}
as well as
\begin{alignat}{1}
- \int \limits_0^{\infty} \into \tilde m (\varphi^N_h) \nabla \mu^N \cdot \nabla \phi \dx \dt =& \int \limits_0^{\infty} \into \partial^-_{t,h} \varphi^N  \phi \dx \dt + \int \limits_0 ^\infty \into \nabla \varphi^N_h \cdot \ve^N \phi \dx \dt , \label{equation_discrete_continuous_2}\\
\int \limits_0^{\infty} \into \mu^N  \phi \dx \dt =&  \int \limits_0^{\infty} \into  \varepsilon \nabla \varphi^N \cdot \nabla \phi \dx \dt + \int \limits_0^{\infty} \into h(q^N) \frac{1}{\varepsilon} H(\varphi^N, \varphi^N_h )  \phi\dx \dt  \nonumber \\\label{equation_discrete_continuous_3}
& + \delta \int \limits_0^\infty \into \partial_{t,h}^- \varphi^N \phi \dx \dt
\end{alignat}
for all $\phi \in C^{\infty}_{(0)} ([0, \infty); C^1 (\overline \Omega))$.

We now derive the energy inequality for the interpolating functions
$\ve^N, q^N, \mu^N$ and $\varphi^N$.
We define the piecewise linear interpolant $E^N (t)$ of $E_{tot} (\ve_k , \varphi_k ,  q_k)$ at $t_k = kh$ by
\begin{align*}
E^N (t) := \frac{(k+1)h - t}{h} E_{tot} (\ve_k , \varphi_k,  q_k ) + \frac{t - kh}{h} E_{tot} (\ve_{k+1}, \varphi_{k+1} , q_{k+1})
\end{align*}
for $t \in [kh , (k+1)h)$. Moreover, we define for all $t \in (t_k , t_{k+1})$, $k \in \mathbb N_0$
\begin{align*}
D^N (t) := & \into m(\varphi_{k}, q_{k}) \left | \nabla q_{k+1} \right | ^2 \dx + \into \tilde m (\varphi_k) |\nabla \mu_{k+1}|^2 \dx + \into 2 \eta (\varphi_k ) |D \ve_{k+1}|^2 \dx  \\
& + \delta \into | \Delta \ve_{k+1}|^2 \dx+ \delta \into \frac{|\varphi_{k+1} - \varphi_k|^2}{h^2} \dx . 
\end{align*}
Thus the time-discrete energy estimate (\ref{equation_discrete_energy_estimate}) yields
\begin{align*}
- \frac{d}{dt} E^N (t) = \frac{E_{tot} (\ve_k, \varphi_k, q_k) - E_{tot} (\ve_{k+1}, \varphi_{k+1}, q_{k+1})}{h} \geq D^N (t) 
\end{align*}
for all $t \in (t_k , t_{k+1})$, $k \in \mathbb N_0$. 
Integrating this inequality, we get the energy estimate for the interpolated functions $\ve^N, q^N, \mu^N$ and $\varphi^N$ given by
\begin{align*}
& E_{tot} \left ( \ve^N_h (t), \varphi^N_h (t), q^N_h (t) \right ) + \int \limits_s^t \into \left ( m(\varphi^N_h, q^N_h)  | \nabla q^N  | ^2   \right . \nonumber \\
& \ \ + \tilde m (\varphi^N_h) | \nabla \mu^N | ^2 + \left . 2 \eta (\varphi^N_h) |D\ve^N|^2 + \delta |\Delta \ve^N|^2 +  \delta  \left | \partial_{t,h}^- \varphi^N \right |^2 \right ) \dx \mathit{d \tau} \nonumber \\
& \leq E_{tot} \left ( \ve^N_h (s), \varphi^N_h (s),  q^N_h (s) \right )
\end{align*}
for all $0 \leq s \leq t < \infty$ with $s, t \in h  \mathbb N_0 $.
This implies:
\begin{enumerate}
\item $ \qquad (\ve^N)_{N \in \mathbb N} \text{ is bounded in } L^\infty (0, \infty; L^2 (\Omega)^d) \cap  L^2 (0, \infty; H^2 (\Omega)^d) $,
\item $ \qquad (\nabla q^N)_{N \in \mathbb N},  (\nabla \mu^N)_{N \in \mathbb N},  ( \partial^-_{t,h} \varphi^N )_{N \in \mathbb N} \text{ are bounded in } L^2 (0,\infty; L^2 (\Omega)) $,
\item $ \qquad  (\nabla \varphi^N)_{N \in \mathbb N} \text{ is bounded in } L^{\infty} (0,\infty; L^2 (\Omega)^d) $,
\item $ \qquad  (W(\varphi^N))_{N \in \mathbb N} \text{ and }  (G(q^N))_{N \in \mathbb N} \text{ are bounded in } L^{\infty} (0,\infty; L^1 (\Omega))$.
\end{enumerate} 
Due to these bounds we can conclude that  there exists a suitable subsequence, which we denote by $(\ve^N, q^N, \mu^N, \varphi^N)_{N \in \mathbb N}$ again, such that
\begin{enumerate}
\item $ \qquad \ve^N \rightharpoonup \ve $ in $L^2 (0, \infty ; H^2 (\Omega)^d )$,
\item $ \qquad \ve^N \rightharpoonup^* \ve $ in $ L^\infty (0, \infty; L^2 (\Omega)^d ) \cong ( L^1 (0, \infty; L^2 (\Omega)^d ))'$,
\item $ \qquad q^N \rightharpoonup q$ in $L^2 (0, T; H^1 (\Omega))$,
\item $ \qquad q^N \rightharpoonup^* q$ in $L^\infty (0 , \infty; L^2 (\Omega)) \cong ( L^1 (0, \infty; L^2 (\Omega)))'$ ,
\item $ \qquad \varphi^N \rightharpoonup^* \varphi$ in $L^{\infty} (0, \infty; H^1 (\Omega)) \cong (L^1 (0,\infty ; H^1 (\Omega)))'$,
\item $ \qquad \mu^N \rightharpoonup \mu$ in $L^2 (0, T; H^1 (\Omega))$
\end{enumerate}
for every $0 < T < \infty$.
Note that in the following we often pass to suitable subsequences $N_k \rightarrow_{k\to \infty} \infty$, which we always denote by $(\ve^N, q^N, \mu^N, \varphi^N)_{N \in \mathbb N}$ again.

Now we want to show that $\varphi $ attains its initial value.
To this end, we denote by $\tilde \varphi^N$ the piecewise linear interpolant of $\varphi^N (t_k)$, where $t_k = kh$, $k \in \mathbb N_0$.
From equation (\ref{equation_discrete_continuous_2}) it follows that $(\partial_t \tilde \varphi^N)_{N \in \mathbb N} \subseteq L^2 (0, \infty ; H^{-1} (\Omega))$ is bounded. Moreover,  $(\tilde \varphi^N)_{N \in \mathbb N}$ is bounded in $L^{\infty} (0, \infty ; H^1 (\Omega))$. Thus we can apply the Aubin-Lions lemma, which yields the relative compactness of $ (\tilde \varphi^N)_{ N \in \mathbb N }$ in $L^p (0,T; L^2 (\Omega))$ for every $0 < T < \infty$ and $1 \leq p < \infty$.
In particular this implies
\begin{align*}
\tilde \varphi^N \rightarrow \tilde \varphi \qquad \text{ in } L^p (0,T; L^2 (\Omega))
\end{align*}
for all $0 < T < \infty$ and $1 \leq p < \infty$, where $\tilde \varphi \in L^{\infty} (0, \infty ; L^2 (\Omega))$. In particular, there exists a subsequence such that $\tilde \varphi^N \rightarrow \tilde \varphi$ pointwise a.e. in $(0, \infty) \times \Omega$. We can even deduce
\begin{align*}
\tilde \varphi^N \rightarrow \varphi \qquad \text{ in } L^2 (0,T; L^2 (\Omega)) 
\end{align*}
for every $0 < T < \infty$ since weak and strong limits coincide.

Since $(\tilde \varphi^N )_{N \in \mathbb N}$ is bounded in $W^1_2 (0, T; H^{-1} (\Omega)) \cap L^2 (0, T ; H^1 (\Omega)) \hookrightarrow C([0,T]; L^2 (\Omega))$, there exists a subsequence such that
\begin{align*}
\tilde \varphi^N \rightharpoonup \varphi \qquad \text{ in } C([0,T]; L^2 (\Omega))
\end{align*}
for every $0 < T < \infty$.
As the mapping $\text{tr}_{t=0} : C ([0,T]; L^2 (\Omega))  \rightarrow L^2 (\Omega) $,  $f  \mapsto f(0)$
is linear and continuous, it is also weakly continuous. Hence, we can deduce $\varphi (0) = \varphi_0$ in $L^2 (\Omega)$.

To prove higher regularity of $\varphi^N$ and $\varphi^N_h$, resp., we use equation (\ref{equation_discrete_continuous_3}), which yields
\begin{align*}
\varepsilon \Delta \varphi^N = \frac{1}{\varepsilon} h(q^N) H(\varphi^N, \varphi^N_h) - \mu^N + \delta \partial^-_{t,h} \varphi^N =: f^N_1 
\end{align*}
in the weak sense. 
Using standard elliptic regularity theory with Neumann boundary condition yields  $\varphi^N (t) \in H^2 (\Omega)$ for a.e. $t \in (0, \infty)$ together with the estimate
\begin{align*}
\|\varphi^N (t)\|_{H^2 (\Omega)} \leq C \left ( \|\varphi^N (t)\|_{H^1 (\Omega)} + \|f^N_1 (t)\|_{L^2 (\Omega)} \right ) \qquad \text{ for a.e. } t \in (0, \infty).
\end{align*}
Since $f^N_1  \in L^2(0, T; L^2 (\Omega))$ is bounded for every $0<T<\infty$, the estimate above implies the boundedness of $\varphi^N , \varphi^N_h \in L^2 (0, T; H^2 (\Omega))$ for every $0<T<\infty$.
Because of
\begin{align*}
\| f \|_{L^\infty ( \Omega)} \leq C \|f\|^{\frac{1}{2}}_{H^1 (\Omega)} \|f\|^{\frac{1}{2}}_{H^2 (\Omega)}\qquad 
\text{for all }f \in H^2 (\Omega),
\end{align*}
we conclude that
$(\varphi^N_h)_{N \in \mathbb N} \text{ is bounded in } L^4 (0, T;
L^\infty (\Omega)) $  and therefore, cf. Theorem 2.32 in
\cite{Dissertation_Weber},
\begin{align*}
(\varphi^N_h)_{N \in \mathbb N} \text{ is bounded in } L^4 (0, T; L^\infty (\Omega)) \cap L^\infty (0,T; L^6 (\Omega)) \hookrightarrow L^{12} (0,T; L^{9} (\Omega)) 
\end{align*}
for every $0 < T < \infty$.
Furthermore, for every $t\in [0,\infty)$ it holds
\begin{align*}
\|\tilde \varphi^N (t) - \varphi^N (t) \|_{L^2 (\Omega)} & \leq C h \|\partial_t \tilde \varphi^N (t) \|_{H^{-1}_0 (\Omega)}^{\frac{1}{2}} \|\tilde \varphi^N (t) - \varphi^N (t)\|_{H^{1} (\Omega)}^{\frac{1}{2}}.
\end{align*}
This implies $\varphi^N \rightarrow \varphi  \text{ in } L^2 (0, T; L^2 (\Omega)) $ for every $0 < T < \infty$.
Due to the interpolation inequality 
\begin{align*}
\|\varphi^N (t) - \varphi (t) \|_{H^1 (\Omega)} \leq C \|\varphi^N (t) - \varphi (t) \|_{H^2 (\Omega)}^{\frac{1}{2}} \|\varphi^N (t) - \varphi (t) \|_{L^2 (\Omega)}^{\frac{1}{2}}
\end{align*}
we obtain $\varphi^N \rightarrow \varphi$ in $L^2 (0, T; H^1 (\Omega))$ for every $0 < T < \infty$.

The proof of the strong convergence of $(q^N)_{N \in \mathbb N}$ in $L^2 (0,T; L^2 (\Omega))$ for a suitable subsequence is a bit more complicated than the proof for $(q^\delta)_{\delta > 0}$ in Theorem \ref{main_theorem_existence_of_weak_solutions}. 
As before, we show the compactness of $( q^N)_{ N \in \mathbb N}$ in $L^2 (Q_T)$ with the aid of Theorem \ref{theorem_simon_5}. 
First, we note that $(q^N)_{N \in \mathbb N}$ is bounded in $L^2 (0,T; H^1 (\Omega))$ and therefore condition i) is fulfilled for every $0 < T < \infty$. For the proof of condition ii) we show
\begin{align}\label{estimate_compactness_q_plus_s_help}
\left ( \int \limits_0^{T-s} \|q^N (t + s) - q ^N (t)\|_{L^2 (\Omega)}^2 \dt \right )^\frac{1}{2} \leq C (T) s^\frac{1}{4}
\end{align}
for all $s = \tilde m h$ with $\tilde m \in \mathbb N$ and a constant $ C(T) > 0$ independent of $s$ and $N \in \mathbb N$. Then \cite[Lemma 9.1]{MR2911126} yields
\begin{align}\label{estimate_compactness_q_plus_lambda_help}
\left ( \int \limits_0^{T-\lambda} \|q^N (t+\lambda) - q^N (t)\|^2_{L^2 (\Omega)} \dt \right ) ^{\frac{1}{2}} \leq C(T) \lambda^\frac{1}{4} 
\end{align}
for any $\lambda > 0$ and a constant $ C(T) > 0$ independent of $\lambda$ and $N$.
This shows condition (ii) of Theorem \ref{theorem_simon_5} and we obtain
the compactness
for the sequence $(q^N)_{N\in \mathbb N}$, Therefore let $s = \tilde mh$ be given for $\tilde m \in \mathbb N$ and $h = \frac{1}{N}$. Moreover, we define
\begin{align*}
\tilde F (\varphi^N, q^N) := \frac{1}{\varepsilon} f(q^N) W(\varphi^N) + g(q^N) , 
\qquad \tilde f (t) := \tilde F (\varphi^N (t), q^N (t)) .
\end{align*}
From the strong monotonicity of $g$ and the monotonicity of $f$ it follows that $\tilde F(\varphi^N, \cdot )$ is strongly monotone as before. Thus there exists a constant $C > 0$ such that
\begin{align*}
\left | \tilde F (\varphi_k , q_{k+\tilde m} ) - \tilde F (\varphi_k , q_k ) \right | \geq C \left | q_{k+ \tilde m} - q_k \right | \geq C \left |q^N (t+ s) - q^N(t) \right |
\end{align*}
for every $t \in [kh, (k+1)h)$.
Multiplying these inequalities with $ |q^N (t+ s) - q^N (t)| $, integrating from $(k-1)h$ to $kh$ with respect to $t$, integrating over the domain $\Omega$ and summing over $k=1,..., TN-m$ for $T \in \mathbb N$ yields
\begin{align}\label{inequality_qNplus_s_qN_estimate}
C  & \sum \limits_{k=1}^{TN- \tilde m}  \into \int \limits_{(k-1)h}^{kh}  |q^N (t+ s) - q^N (t)|^2 \dt \dx \leq 
C \sum \limits_{k = 1}^{TN - \tilde m} \into \int \limits_{(k-1)h}^{kh}|q_{k+\tilde m} - q_k|^2 \dt \dx  \nonumber \\
& \leq  \sum \limits_{k=1}^{TN - \tilde m} \into \int \limits_{(k-1)h}^{kh} ( \tilde F (\varphi_k , q_{k+ \tilde m}) - \tilde F (\varphi_{k+ \tilde m} , q_{k+ \tilde m}) ) (q_{k+ \tilde m} - q_k ) \dt \dx \nonumber \\
& \ \ \ \ +  \sum \limits_{k=1}^{TN - \tilde m} \into \int \limits_{(k-1)h}^{kh} ( \tilde F (\varphi_{k+ \tilde m} , q_{k+ \tilde m}) - \tilde F (\varphi_k , q_k) ) (q_{k+ \tilde m} - q_k ) \dt \dx .
\end{align}
Since $f$ is a bounded function, we can conclude for the first term in  (\ref{inequality_qNplus_s_qN_estimate})
\begin{align*}
 &\sum \limits_{k=1}^{TN - \tilde  m} \into \int \limits_{(k-1)h}^{kh} ( \tilde F (\varphi_k , q_{k+ \tilde m}) - \tilde F (\varphi_{k+ \tilde m} , q_{k+ \tilde m}) ) (q_{k+ \tilde m} - q_k ) \dt \dx \\
& \  \ \ = \int \limits_0^{T-s}  \into   \left(  \tilde F (\varphi^N , q^N_{s+}) - \tilde F (\varphi^N _{s+}, q^N _{s+}) \right ) \left ( q^N _{s+} - q^N  \right ) \dx \dt  \\
& \ \ \  \leq C \int \limits_0^{T-s} \into \left | \left ( \varphi^N _{s+} - \varphi^N  \right )  \left ( |\varphi^N |^2 + |\varphi^N _{s+}|^2 +1 \right ) \right | \  \left | q^N _{s+} - q^N \right | \dx \dt \leq C (T) s^{\frac{1}{4}} ,
\end{align*}
where we used $\varphi^N \in L^\infty (0, \infty; L^6 (\Omega))$, $q^N_{s+}$, $q^N \in  L^2 (0,T-s; L^6 (\Omega))$ and 
\begin{align*}
\underset{0 \leq t \leq T-s}{\sup} \|\varphi^N (t+s) - \varphi^N (t)\|_{L^2 (\Omega)} \leq C(T) s^{\frac{1}{4}} \qquad \text{ a.e. in } (0,\infty)
\end{align*}
for every $0 < T < \infty$ and a constant $C(T) > 0$ depending on $T$. The proof for the latter inequality is similar to the proof of (\ref{delta_inequality_varphi_N_difference_l2}), where we use that there exists $k \in \mathbb N$ such that $t \in [kh, (k+1)h )$ and $t+s \in [(k+ \tilde m)h, (k+ \tilde m+1)h)$. 
From $\varphi^N (t) = \varphi_{k+ \tilde m+1}$ for  $t \in [(k+ \tilde m)h, (k+ \tilde m+1) h)$ we can deduce
\begin{align*}
\varphi^N (t+s) - \varphi^N (t) &= \varphi_{k+ \tilde m+1} - \varphi_{k+1}  = \tilde \varphi^N ((k+ \tilde m+2) h) - \tilde \varphi^N ((k+2) h) \\
&= \tilde \varphi^N ( \tilde t + s) - \tilde \varphi^N (\tilde t),
\end{align*}
where $\tilde t := (k+2) h$. Since $(\tilde \varphi^N)_{N \in \mathbb N} $ is bounded in $L^\infty (0 , T ; H^1 (\Omega)) \cap W^1_2 (0, T ; H^{-1 }_0 (\Omega))$ for all $0<T<\infty$, the statement follows similarly as in (\ref{delta_inequality_varphi_N_difference_l2}).

For the second term in (\ref{inequality_qNplus_s_qN_estimate}), we set  $l (t) := \left \lfloor \frac{t}{h} \right \rfloor$ and $\tilde t (t) := h  \left \lfloor \frac{t}{h} \right \rfloor $. Then it holds $\tilde t (t) = t_k $ for $t \in [kh, (k+1)h)$. Hence, we have
\begin{align*}
&  \sum \limits_{k=1}^{TN - \tilde m} \into \int \limits_{(k-1)h}^{kh} ( \tilde F (\varphi_{k+ \tilde m} , q_{k+ \tilde m}) - \tilde F (\varphi_k , q_k) ) (q_{k+ \tilde m} - q_k ) \dt \dx  \\
& = \sum \limits_{k=1}^{TN - \tilde m} \into \int \limits_{(k-1)h}^{kh} \left ( \sum \limits_{j=1}^{\tilde m} \tilde f (\tilde t (t)+ jh) - \tilde f ( \tilde t (t)+ (j-1) h)  \right ) (q_{k+ \tilde m} - q_k )   \dt \dx \\
& = \int \limits_0^{T-s} \into  \left ( \sum \limits_{j=1}^{\tilde m} \tilde f ( \tilde t (t) + jh) - \tilde f ( \tilde t (t)+ (j-1) h)  \right ) (q_{k+ \tilde m} - q_k )  \dx \dt .
\end{align*}
Using equation (\ref{equation_time_discrete_1}) gives
\begin{align*}
 &\sum \limits_{j=1}^{\tilde m} \tilde f (\tilde t (t)+ jh) - \tilde f ( \tilde t (t) + (j-1) h)  \\
&=   h  \sum \limits_{j=l}^{l+ \tilde m-1}  \left(\di \left ( m (\varphi_j, q_j) \nabla q_{j+1} \right )  -  \nabla \left ( \frac{1}{\varepsilon} f (q_{j+1}) W(\varphi_j) + g(q_{j+1}) \right ) \cdot \ve_{j+1}\right) \\
&= \int \limits_{\tilde t(t)}^{\tilde t(t)+s} \di \left ( m (\varphi^N (\tau), q^N (\tau) ) \nabla q^N (\tau + h) \right ) \mathit{d \tau}   - \int \limits_{\tilde t(t)}^{\tilde t(t)+s}  \nabla g (q^N (\tau + h)) \cdot \ve^N (\tau + h) \mathit{ d \tau} \\
& \ \ \ \ -  \int \limits_{\tilde t(t)}^{\tilde t(t)+s}  \di  \left ( \frac{1}{\varepsilon} f (q^N (\tau + h)) W (\varphi^N (\tau )) \ve^N (\tau + h )\right ) \mathit{d \tau} 
\end{align*}
in $H^{-1}_0 (\Omega) $. 
We use this identity in (\ref{inequality_qNplus_s_qN_estimate}) and obtain
\begin{align*}
& \int \limits_0^{T-s} \into  \left ( \tilde F (\varphi^N _{s+}, q^N _{s+}) - \tilde F (\varphi^N , q^N)  \right )   \left ( q^N _{s+} - q^N  \right ) \dx \dt \\
& \leq  \int \limits_0^{T-s} \into \int \limits_{\tilde t(t)}^{\tilde t(t)+s}  \left  | m (\varphi^N (\tau), q^N (\tau) ) \nabla q^N (\tau + h)  \right | \mathit{d \tau} |\nabla q^N_{s+} - \nabla q^N| \dx \dt \\
& \ \ \  \ +   \int \limits_0^{T-s} \into   \int \limits_{\tilde t(t)}^{\tilde t(t)+s}  \left | g (q^N (\tau + h) \ve^N (\tau + h) \right | \mathit{ d \tau}   |\nabla q^N_{s+} - \nabla q^N| \dx \dt \\
& \ \ \ \ + \int \limits_0^{T-s} \into  \int \limits_{\tilde t(t)}^{\tilde t(t)+s}  \left | \frac{1}{\varepsilon} f (q^N (\tau + h)) W (\varphi^N (\tau )) \ve^N (\tau + h )\right | \mathit{d \tau}  |\nabla q^N_{s+} - \nabla q^N| \dx \dt .
\end{align*}
In the following, we study these three terms separately. The first term can be estimated by
\begin{align*}
& \int \limits_0^{T-s} \into  \int \limits_{\tilde t(t)}^{\tilde t(t)+s}  \left  | m (\varphi^N (\tau), q^N (\tau) ) \nabla q^N (\tau + h)  \right | \mathit{d \tau} \left | \nabla q^N_{s+} - \nabla q^N \right | \dx \dt \\
& \leq C \int \limits_0^{T-s} s^\frac{1}{2} \|q^N\|_{L^2 (0,T; H^1 (\Omega))}   \|\nabla q^N (t+s) - \nabla q^N (t)\|_{L^2 (\Omega)} \dt   \leq C (T) s^\frac{1}{2}  ,
\end{align*}
where $C (T) > 0$.
Using the boundedness of $f$ and the growth condition for $W$, we estimate the third term, similarly as before, by
\begin{align*}
& \int \limits_0^{T-s} \into  \int \limits_{\tilde t(t)}^{\tilde t(t)+s}  \left | \frac{1}{\varepsilon} f (q^N (\tau + h)) W (\varphi^N (\tau )) \ve^N (\tau + h )\right | \mathit{d \tau} \left | \nabla q^N_{s+} - \nabla q^N \right | \dx \dt  \\
& \leq C \int \limits_0^{T-s} s^\frac{1}{4} ( \|\varphi^N\|_{L^{12} (0,T; L^{9} (\Omega))}^3 + 1 ) \|\ve^N\|_{L^2 (0,T; L^6 (\Omega))} \|\nabla q^N (t + s) - \nabla q^N (t)\|_{L^2 (\Omega)} \dt  \\
& \leq C (T) s^\frac{1}{4} .
\end{align*}
For the second term we get in the same way
\begin{align*}
& \int \limits_0^{T-s} \into   \int \limits_{\tilde t(t)}^{\tilde t(t)+s}  \left | g (q^N (\tau + h) \ve^N (\tau + h) \right | \mathit{ d \tau}   \left | \nabla q^N_{s+} - \nabla q^N \right | \dx \dt   \leq C (T) s^\frac{1}{4} .
\end{align*}
Using these estimates it follows that
(\ref{estimate_compactness_q_plus_s_help}) holds and therefore (\ref{estimate_compactness_q_plus_lambda_help}) holds for every $\lambda \in (0,T)$. Thus $(q^N )_{N \in \mathbb N }$ is relatively compact in $L^2 (0,T; L^2 (\Omega))$ and there exists a subsequence such that $q^N \rightarrow \tilde q  \text{ in } L^2 (0, T ; L^2 (\Omega))  $ for all $0<T<\infty$.
In particular, it holds for a subsequence $q^N (t,x) \rightarrow q (t,x)  \text{ a.e. in } (0 , \infty) \times \Omega $.

Now we can show with similar arguments as in \cite{MR3084319} and \cite[Section 2.1]{MR1422251} that  $\ve^N \rightarrow \ve$ in $L^2 (0,T; L^2 (\Omega)^d)$ for all $0 < T < \infty$ as $N \rightarrow \infty$. Thus we can deduce the pointwise convergence a.e. in $(0,T) \times \Omega$. 
Here we can use that from (\ref{equation_discrete_continuous_0}) it follows that $\partial_t ( \mathbb P_\sigma ( \widetilde{ \rho \ve}^N))$ is bounded in $L^1 (0,T; V' (\Omega))$. More precisely, we have the following bounds:
\begin{align*}
\rho^N_h \ve^N \otimes \ve^N & \text{ is bounded in } L^2 (0,T; L^{\frac{3}{2}}(\Omega)^{d \times d}) , \\
\ve^N \otimes \Je^N & \text{ is bounded in } L^\frac{4}{3} (0,T; L^\frac{6}{5} (\Omega)), \\
\mu^N \nabla \varphi^N_h & \text{ is bounded in }  L^2 (0,T; L^{\frac{3}{2}}(\Omega)^d) , \\
\frac{h(q^N)}{\varepsilon} W' (\varphi^N_h) \nabla \varphi^N_h & \text{ is bounded in } L^{\frac{4}{3}} (0,T; L^{\frac{6}{5}} (\Omega)^d)
\end{align*}
for all $0<T<\infty$.
The proof is similar to the one in Theorem \ref{main_theorem_existence_of_weak_solutions}.
Therefore we can use for these terms in (\ref{equation_discrete_continuous_0}) test functions $\boldsymbol \psi \in L^1 (0,T; W^1_6 (\Omega)^d)$. It remains to show
\begin{align*}
 \left|\left \langle \frac{\tilde R^N \ve^N}{2} , \boldsymbol \psi \right \rangle\right| &= \left|\frac{1}{2} \int \limits_0^T \into \frac{\rho^N - \rho^N_h}{h} \ve^N \cdot \boldsymbol \psi \dx \dt  -  \frac{1}{2}   \int \limits_0^T \into \left (  \rho^N_h \ve^N  + \Je^N \right ) \cdot \nabla ( \ve^N  \cdot \boldsymbol \psi )   \dx \dt\right|  \\
& \leq C \|\boldsymbol \psi\|_{L^\infty (0,T; H^2 (\Omega))}.
\end{align*}
Here, we use $\frac{\rho (\varphi^N) - \rho (\varphi^N_h)}{h} \rightharpoonup \rho ' (\varphi) \partial_t \varphi$ in $L^2 (Q_T)$ and $\|\nabla ( \ve^N \cdot \boldsymbol \psi )\|_{L^2(\Omega)}\leq C\|\ve^N\|_{H^1(\Omega)}\|\psi\|_{H^2(\Omega)}$. Thus we can conclude that $(\mathbb P_\sigma ( \widetilde{\rho \ve}^N))_{N \in \mathbb N}$ is bounded in $L^2 (0,T; H^1 (\Omega)^d) \cap W^1_1 (0,T; V (\Omega)')$ and a variant of the Lemma of Aubin-Lions, cf.~\cite[Corollary~5]{MR916688},  yields the existence of $\boldsymbol \omega \in L^2 (0,T; L^2_\sigma (\Omega))$ such that $\mathbb P_\sigma (\widetilde{\rho \ve}^N ) \rightarrow \boldsymbol \omega$ in $L^2 (0,T; L^2_\sigma (\Omega))$. Using the definition of the linear interpolants $\mathbb P_\sigma (\widetilde{\rho \ve}^N)$, one can show $\mathbb P_\sigma ( \widetilde{\rho \ve}^N ) \rightharpoonup \mathbb P_\sigma ( \rho \ve )$ in $L^2 (0,T; L^2_\sigma (\Omega))$. With this, we can prove $\ve^N \rightarrow \ve$ in $L^2 (Q_T)$
as in Section~\ref{section_existence_of_a_weak_solution}.
Note that one can even show
\begin{align*}
\ve^N \rightarrow \ve \qquad \text{ in } L^q (0,T; L^\infty (\Omega)^d) \text{ for every } 1 \leq q <  \frac{8}{3} , 0<T<\infty
\end{align*}
using $\ve^N \rightarrow \ve$ in $L^p(0,T; L^2 (\Omega)^d)$ for every $1 \leq p < \infty$ and the boundedness of $\ve^N $ and $ \ve$ in $L^2 (0,T; H^2 (\Omega)^d)$.

Finally, we pass to the limit $N \rightarrow \infty$  in  \eqref{equation_discrete_continuous_0}-\eqref{equation_discrete_continuous_3}  and prove that $(\ve, \varphi, \mu, q)$ is a weak solution of \eqref{equation_delta_model_0}-\eqref{equation_delta_model_3} in the sense of Definition \ref{definition_weak_solution}. 
It holds
\begin{align*}
\int \limits_0^\infty \into \mu^N \nabla \varphi^N_h \cdot \boldsymbol \psi \dx \dt 
& \rightarrow \int \limits_0^\infty \into \mu \nabla \varphi \cdot \boldsymbol \psi \dx \dt 
\end{align*}
for all  $\boldsymbol \psi \in C^\infty_{(0)} ([0, \infty) ; V (\Omega) )$ due to $\mu^N \rightharpoonup  \mu$ in  $L^2 (0,T; H^1 (\Omega))$ for all $0<T<\infty$ and 
\begin{align*}
\varphi^N_h & \rightarrow \varphi, \varphi^N  \rightarrow \varphi\quad \text{ in } L^p (0,T; H^1 (\Omega))
\end{align*}
for all $1\leq p<\infty$, $0<T<\infty$.
Here the last two convergences hold due to the interpolation result
\begin{align*}
\|\nabla \varphi^N - \nabla \varphi\|_{L^p (0,T; L^2 (\Omega))} & \leq \|\nabla \varphi^N - \nabla \varphi\|_{L^\infty (0,T; L^2 (\Omega))}^{1  - \theta} \|\nabla \varphi^N - \nabla \varphi\|_{L^2 (0,T; L^2 (\Omega))}^\theta
\end{align*}
for every $\theta \in (0,1)$ and $\frac{1}{p} = \frac{\theta}{2}$.

For the proof of
\begin{align*}
\left \langle \frac{\tilde R^N \ve^N}{2} , \boldsymbol \psi \right \rangle \rightarrow_{N\to\infty} \left \langle \frac{\tilde R\ve}{2} ,\boldsymbol \psi \right \rangle
\end{align*}
for all $\boldsymbol \psi \in C^\infty_{(0)} ([0,\infty);  V (\Omega) )$ we use $\ve^N \rightarrow \ve \text{ in } L^p (0,T; L^4 (\Omega)^d) \text{ for all } 1 \leq p < \frac{8}{3}$ and $\ve^N \otimes \ve^N \rightarrow \ve \otimes \ve  \text{ in } L^q (0,T; L^2 (\Omega)^{d \times d}) \text{ for every } 1 \leq q < \frac{4}{3}$ and all $0<T<\infty$ as well as
\begin{equation*}
  \partial_{t,h}^-\varphi^N\rightharpoonup_{N\to\infty} \partial_t \varphi \qquad \text{in }L^2((0,\infty)\times \Omega)
\end{equation*}
for a suitable subsequence.

In (\ref{equation_discrete_continuous_3}) we obtain from the growth condition for $h$ 
\begin{align*}
\left | h(q^N) \frac{1}{\varepsilon} H(\varphi^N, \varphi^N_h) \right | \leq \frac{C}{\varepsilon} (|q^N| + 1)    \left ( |\varphi^N (t)|^2 + |\varphi^N (t-h)|^2 + 1 \right )  ,
\end{align*}
for a.e. $(t,x) \in (0, \infty ) \times \Omega$. From the boundedness of $q^N \in L^2 (0, T; L^6 (\Omega))$ and $\varphi^N , \varphi^N_h \in L^4 (0,T; L^\infty (\Omega))$  we can deduce the boundedness of $h(q^N) \frac{1}{\varepsilon} H(\varphi^N, \varphi^N_h)$  in $L^{\frac{4}{3}} (0,T; L^6 (\Omega))$ for all $0<T<\infty$.
It holds 
\begin{align*}
h(q^N) \frac{1}{\varepsilon} H(\varphi^N, \varphi^N_h)  = \begin{cases} h(q^N)  \frac{1}{\varepsilon} W' (\varphi^N (t-h)) & \text{ if } \varphi^N (t) = \varphi^N (t-h) , \\  h(q^N) \frac{1}{\varepsilon} W' (\xi_ {t, t-h} (x))  & \text{ if } \varphi^N (t) \neq \varphi^N (t-h) , \end{cases}
\end{align*}
where $\xi _{t,t-h}(x) \in [\varphi^N (t-h, x), \varphi^N (t,x)] = [\varphi_k (x) , \varphi_{k+1} (x)]$ for $k \in \mathbb N_0$ such that $t \in [kh, (k+1)h)$.
Since it holds $\varphi^N (t,x) \rightarrow \varphi (t,x)$ a.e and $\varphi^N (t-h, x)  \rightarrow \varphi(t,x)$ a.e., it follows
\begin{align*}
h(q^N) \frac{1}{\varepsilon} H(\varphi^N, \varphi^N_h)_{|(t,x)} \rightarrow h(q(t,x)) \frac{1}{\varepsilon} W' (\varphi(t,x)) \qquad \text{ a.e. in } (0, \infty) \times \Omega
\end{align*}
as $N \rightarrow \infty$.
However, these terms are bounded e.g. in $L^{\frac{4}{3}} ((0, T) \times \Omega)$. Hence, 
\begin{align*}
h(q^N) \frac{1}{\varepsilon} H(\varphi^N, \varphi^N_h) \rightarrow h(q) \frac{1}{\varepsilon} W' (\varphi)  \quad \text{ in } L^p ((0, T) \times \Omega) , \ 1 \leq p < \frac{4}{3}, \ 0 < T < \infty 
\end{align*}
and therefore
\begin{align*}
\int \limits_0^\infty \into h(q^N) \frac{1}{\varepsilon}H(\varphi^N, \varphi^N_h) \phi \dx \dt \rightarrow \int \limits_0^\infty \into h(q) \frac{1}{\varepsilon} W' (\varphi) \phi \dx \dt
\end{align*}
for all $\phi \in C_{(0)}^\infty ([0, \infty) ; C^1 (\overline \Omega))$.

The energy inequality in the case $\delta > 0$ can be proven as in the proof of Theorem \ref{main_theorem_existence_of_weak_solutions} again.
\end{proof}


\bibliographystyle{abbrv}

\end{document}